\documentclass[12pt]{article}
\usepackage{amsfonts,graphicx,xspace}
\usepackage{amsthm,xcolor}
\usepackage{fullpage}
\newtheorem{remark}{Remark}[section]
\newtheorem{theorem}{Theorem}[section]
\newtheorem{lemma}{Lemma}[section] 
\newtheorem{proposition}{Proposition}[section] 
\newtheorem{definition}{Definition}[section]
\newtheorem{corollary}{Corollary}[section]
\newcommand{\RR}{\mathbb{R}}
\newcommand{\R}{\RR}

\newcommand{\Sn}{\mathbb{S}_{++}^n}

\newcommand{\gr}{\nabla_{_H}}
\newcommand{\F}{{{\cal F}}}
\newcommand{\A}{{\cal A}}
\newcommand{\infi}{\rightarrow +\infty}
\newcommand{\T}{\textrm}
\newcommand{\N}{{\cal N}}
\newcommand{\<}{\langle}
\newcommand{\HS}{(H\mbox{-}SD)}
\newcommand{\ra}{\rangle}
\newcommand{\me}{(\cdot,\cdot)}

\def\Ker{{\rm Ker}\:}
\def\Im{{\rm Im}\:}
\def\dom{{\rm dom}\:}
\def\inte{{\rm int}\:}
\def\inteA{{\rm int}_\A}
\def\ri{{\rm ri}\:}
\def\reels{\RR}
\def\argmin{\mbox{\rm Argmin}}
\def\eps{\varepsilon}
\def\bd{{\rm bd}\:}

\title{\Large Hessian Riemannian gradient flows in convex programming\footnote{\color{purple}Submitted in December 2012, published in 2004: SIAM J. CONTROL OPTIM., Vol. 43, No. 2, pp. 477--501 }}
\author{Felipe  Alvarez\thanks{Departamento de Ingenier\'\i a Matem\'atica
and Centro de Modelamiento Matem\'atico (CNRS UMR 2071),
Universidad de Chile, Blanco Encalada 2120, Santiago, Chile. {\tt falvarez@dim.uchile.cl}. Supported by Fondecyt 1020610,
Fondap en Matem\'aticas Aplicadas and Programa Iniciativa Cient\'\i fica
Milenio.} \and J\'er\^ome Bolte\thanks{ACSIOM-CNRS FRE 2311, D\'epartement de Math\'ematiques, case 51, Universit\'e Montpellier II, Place Eug\`ene Bataillon, 34095 Montpellier cedex 5, France. Partially supported by Ecos-Conicyt C00E05.} \and Olivier Brahic.\thanks{GTA-CNRS UMR 5030, D\'epartement de Math\'ematiques, case 51, Universit\'e Montpellier II, Place Eug\`ene Bataillon, 34095 Montpellier cedex 5, France.}}
\date{}

\begin{document}
\maketitle
{\bf Abstract} Motivated by a constrained minimization problem, it is studied the gradient
flows with respect to Hessian Riemannian metrics induced by
convex functions of Legendre type.  The first result
characterizes Hessian Riemannian structures on convex sets as
those metrics that have a specific integration property with
respect to variational inequalities, giving a new motivation for
the introduction of Bregman-type distances. Then, the general
evolution problem is introduced and a differential inclusion
reformulation is given.  A general existence result is proved and
global convergence is established under quasi-convexity
conditions, with interesting refinements in the case of convex
minimization. Some explicit examples of these gradient
flows are discussed. Dual trajectories are identified and sufficient
conditions for dual convergence are examined  for a convex
program with positivity and equality constraints. Some convergence
rate results are established. In the case of a linear objective
function, several optimality characterizations of the orbits are
given: optimal path of viscosity methods, continuous-time model
of Bregman-type proximal  algorithms, geodesics for some adequate
metrics and projections of $\dot q$-trajectories of some Lagrange equations and completely integrable Hamiltonian systems. \\



\thispagestyle{plain} 
\markboth{F. ALVAREZ, J. BOLTE \& O. BRAHIC}{HESSIAN RIEMANNIAN GRADIENT FLOWS}

\section{Introduction}
The aim of this paper is to study the existence, global convergence and 
geometric properties of gradient flows with respect to a specific class of 
Hessian Riemannian metrics on convex sets. Our work is indeed  deeply 
related to the constrained minimization problem
$$
\min \{f(x)\mid x\in \overline{C},\; Ax=b\}, \leqno (P)
$$
where $\overline{C}$ is the closure of a nonempty, {\it open}  and convex 
subset $C$ of $\R^n$, $A$ is a $m< n$
real matrix with $m\leq n$,  $b\in \RR^m$ and $f\in C^1(\R^n)$. A strategy 
to solve $(P)$ consists in endowing $C$ with a Riemannian structure $\me^H$, 
to
restrict it to the relative interior of the feasible set $\F:=C\cap\{x\mid 
Ax=b\}$, and then to consider the trajectories generated by the steepest 
descent vector field $-\gr f_{|_{\F}}$. This leads to the initial value 
problem
$$\HS\qquad \dot x(t)+\gr f_{|_{\F}}(x(t))=0,\:x(0)\in \F,$$
where $\HS$ stands for $H$-steepest descent. We focus  on those metrics that 
are induced by the Hessian $H=\nabla^2 h$ of a {\em Legendre type} convex 
function $h$ defined on $C$ (cf. Def. \ref{D:legendre}).

The use of Riemannian methods in optimization has increased
recently: in relation with Karmarkar algorithm and
linear programming  see Karmarkar \cite{Kar90}, Bayer-Lagarias \cite
{BaL89}; for continuous-time models of proximal type algorithms and related 
topics see  Iusem-Svaiter-Da Cruz \cite{IuS99}, Bolte-Teboulle \cite{BoT02}. 
For a systematic dynamical system approach to constrained optimization based 
on double bracket flows, see  Brockett \cite{Bro88,Bro91}, the  monograph of 
Helmke-Moore \cite{HeM94} and the references therein. On the other hand, the 
structure of $\HS$ is also at the heart of some important problems in 
applied
mathematics. For connections with population dynamics and game theory see 
Hofbauer-Sygmund \cite{HoS98}, Akin \cite{Aki79}, Attouch-Teboulle 
\cite{AttTeb}. We will see that  $\HS$ can be reformulated as the 
differential inclusion $\frac{d}{dt}\nabla h(x(t))+\nabla f(x(t))\in \Im
A^T,\:x(t)\in\F,$ which is formally similar to some evolution problems in 
infinite dimensional spaces arising in thermodynamical systems, see for 
instance Kenmochi-Pawlow \cite{KenPaw} and references therein.

   A classical approach in the asymptotic analysis of
dynamical systems consists in exhibiting attractors
of the orbits by using Liapounov functionals. Our choice of Hessian 
Riemannian metrics is based on this idea.  In fact, we consider first the 
important case where $f$ is convex, a condition that permits us to 
reformulate  $(P)$ as a variational inequality problem: $\mbox{find }a\in 
\overline{\F} \mbox{ such that }(\gr
f_{|_{\F}}(x),x-a)^H_x\geq0\:\mbox{ for all }x\mbox{ in }\F.$ In order to 
identify a suitable Liapounov functional, this variational problem  is met 
through the
following
integration problem: {\it find the metrics $\me^H$ for which the vector 
fields $V^a:\F\to\RR^n$, $a\in \F$, defined by  $V^a(x)=x-a,$
are $\me^H$-gradient vector fields}.
Our first result (cf. Theorem \ref{T:Poincare}) establishes that such 
metrics are  given by the Hessian of strictly convex functions, and in that 
case
the vector fields $V^a$ appear  as gradients with respect to the
second variable of some distance-like functions that are called 
$D$-functions. Indeed, if
$\me^H$ is induced by the Hessian $H=\nabla^2 h$ of $h: \F\mapsto \R$, we 
have for all
$a,x$ in $\F$: $\gr D_h(a,.)(x)=x-a, \mbox{ where }D_h(a,x) 
=h(a)-h(x)-dh(x)(x-a).$ For another characterization
of Hessian metrics,  see
Duistermaat \cite{Dui01}.

Motivated by the previous result and with the aim of solving $(P)$, we are 
then naturally led to consider Hessian Riemannian metrics that cannot be 
smoothly
extended out of $\F$. Such a requirement is fulfilled by the Hessian of a  
 {\it Legendre  (convex) function } $h$, whose definition is recalled in section 
\ref{S:gradientflow}. We give then a differential inclusion
reformulation of $\HS$, which permits to show that in the case of a linear 
objective function $f$, the flow of $-\gr f_{|_{\F}}$
stands at the crossroad of many optimization methods. In fact,  following 
\cite{IuS99}, we prove that
viscosity methods and Bregman proximal algorithms produce their paths or iterates in the orbit of $\HS$. The 
$D$-function of $h$ plays an essential role for this. In section 
\ref{S:Examples} it is given  a systematic method to construct Legendre 
functions based on barrier functions for convex inequality problems, which 
is illustrated with some examples; relations to other works are discussed.

Section \ref{GEAS} deals with global existence and convergence properties. 
After having given a non trivial  well-posedness result  (cf. Theorem 
\ref{T:existence}), we prove in section \ref{S:value} that 
$f(x(t))\rightarrow \inf_{\overline{\F}}f$ as $t\infi$ whenever $f$ is 
convex.   A natural problem that arises is the trajectory  convergence to a 
critical point. Since one expects the limit  to be a (local) solution to 
$(P)$, which may belong to the boundary of $C$, the notion of critical point 
must be understood in the sense of the  optimality condition for a local 
minimizer $a$  of $f$ over $\overline{\F}$:
$$({\cal O})\qquad \nabla f(a)+N_{\overline{\F}}(a)\ni 0,\:a\in
\overline{\F},$$
where $N_{\overline{\F}}(a)$  is the normal cone to $\overline{\F}$ at $a$, 
and $\nabla f$ is the Euclidean gradient of $f$. This involves an asymptotic 
singular  behavior that is rather unusual in the classical theory of 
dynamical systems, where the critical points are typically supposed to be in 
the manifold.
In section \ref{S:bregman} we assume that the Legendre type function $h$ is 
a {\em Bregman function with zone $C$} and prove that under a 
quasi-convexity assumption on $f$, the trajectory converges to some point 
$a$ satisfying $({\cal O})$. When $f$ is convex,  the preceding result 
amounts to the convergence of $x(t)$ toward a global minimizer of $f$ over 
$\overline{\F}$. We also give a variational characterization of the limit 
and establish an abstract result on the rate of convergence under uniqueness 
of the solution. We consider in section \ref{S:LP} the case of linear 
programming, for which asymptotic convergence as well as a variational 
characterization are proved without the Bregman-type condition.  Within this 
framework, we also give some estimates on the convergence rate that are 
valid for  the specific Legendre functions commonly used in practice.
In section \ref{S:dual}, we consider the interesting  case of positivity and 
equality constraints, introducing a {\em dual} trajectory $\lambda(t)$ that, 
under some appropriate conditions,  converges to a solution to the dual 
problem of $(P)$ whenever $f$ is convex, even if primal convergence is not 
ensured.

Finally, inspired by the seminal work \cite{BaL89}, we define in section 
\ref{S:legendre-transform} a change of coordinates called {\it Legendre 
transform coordinates}, which permits to show that the orbits of $\HS$ may 
be seen as straight lines in a positive cone. This leads to additional 
geometric interpretations of the flow of $-\gr
f_{|_{\F}}$. On the one hand, the orbits
are geodesics with respect to an appropriate metric and,
on the other hand, they may be seen as $\dot q$-trajectories of
some Lagrangian, with consequences in terms of integrable Hamiltonians.

  {\bf Notations.} $\Ker A=\{x\in\reels^n\;|\; Ax=0\}.$
The orthogonal complement of $\A_0$ is denoted by $\A_0^\perp$, and
$\langle \cdot , \cdot \rangle$ is the standard Euclidean scalar product of
$\R ^n$. Let us denote by $\Sn$ the cone of real symmetric definite positive
matrices. Let $\Omega\subset\R^n$ be an open set. If $f:\Omega\to \R$ is
differentiable then $\nabla f$ stands for the  Euclidean gradient of $f$. If
$h:\Omega\mapsto\R$
is twice differentiable  then its Euclidean Hessian at $x\in \Omega$ is
denoted by $\nabla^2 h(x)$ and is defined as the endomorphism of $\R^n$
whose matrix in canonical coordinates is given by $[\frac{\partial^2
h(x)}{\partial x_i\partial x_j}]_{i,j \in \{1,..,n\} }$. Thus,  $\forall
x\in \Omega$, $d^2h(x)=\<\nabla^2h(x)\:\cdot,\cdot\ra$.



\section{Preliminaries}

\subsection{The minimization problem and optimality
conditions}\label{S:problem}

Given a positive integer $m< n$,
a full rank matrix $A\in\RR^{m\times n}$ and $b\in \Im A$, let us define
\begin{equation}\label{E:affine-space}
\A=\{x\in\reels^n\;|\; Ax=b\}.
\end{equation}
Set $\A_0=\A-\A=\Ker A$.  Of course, $\A_0^\perp=\Im A^T$ where $A^T$ is the
transpose of $A$.  Let $C$ be a nonempty, open and convex subset of $\RR^n$,
and
$f:\reels^n\to \reels$ a ${\cal C}^1$ function. Consider the constrained
minimization problem
$$
\inf \{f(x)\;|\;x\in \overline{C},\; Ax=b\}. \leqno (P)
$$
The set of optimal solutions of
$(P)$ is denoted by $S(P)$.  We call $f$ the {\em objective
function} of $(P)$. The {\em feasible set}
of $(P)$ is given by $
\overline{\F}=\{x\in\RR^n\;|\;x\in \overline{C},\;
Ax=b\}=\overline{C}\cap \A,$
and $\F$ stands for the {\em relative interior} of
$\overline{\F}$, that is
\begin{equation}\label{E:feasible}
\F=\ri \overline{\F}=\{x\in\RR^n\;|\;x\in C,\;
Ax=b\}=C\cap\A.
\end{equation}
Throughout this article, we assume that
\begin{equation}\label{E:hypof}
\F\neq\emptyset .
\end{equation}
It is well known that a necessary  condition for $a$ to be
locally minimal for $f$ over $\overline{\F}$ is \( ({\cal O}):\:-\nabla 
f(a)\in N_{\overline{\F}}(a)\), where $N_{\overline{\F}}(x)=\{\nu \in
\reels^n\;|\; \forall
y\in\overline{\F},\; \langle y-x,\nu\rangle \leq 0\}$ is the {\em normal
cone} to $\overline{\F}$ at $x\in\overline{\F}$
($N_{\overline{\F}}(x)=\emptyset$ when
$x\notin\overline{\F}$); see for instance \cite[Theorem
6.12]{RoW98}.  By \cite[Corollary 23.8.1]{Roc70}, \(
N_{\overline{\F}}(x)=N_{\overline{C} \cap
\A}(x)=N_{\overline{C}}(x)+N_{\A}(x)=N_{\overline{C}}(x)+\A_0^\perp,\) for
all \( x\in\overline{\F}\). Therefore, the necessary optimality condition
for $a\in \overline{\F}$ is
\begin{equation}\label{E:optcond}
-\nabla f(a)\in N_{\overline{C}}(a)+ \A_0^\perp.
\end{equation}
If  $f$ is convex then this condition  is also sufficient for
$a\in\overline{\F}$ to be in $S(P)$.




\subsection{Riemannian gradient flows on the relative interior of the
feasible set}

Let $M$ be a smooth manifold. The tangent space to $M$ at $x\in M$ is denoted by $T_x M$. If
$f:M\mapsto \R$ is a ${\cal C}^1$ function then $df(x)$ denotes its
differential or tangent  map  $df(x):T_xM\to \RR $ at $x\in M$. A
${\cal C}^k$ metric on $M$, $k\geq 0$, is a family of scalar products
$(\cdot ,\cdot)_x$  on each $T_x M$, $x\in M$, such that $(\cdot
,\cdot)_x$ depends in a ${\cal C}^k$ way on $x$. The couple $M, (\cdot
,\cdot)_x$ is called a ${\cal C}^k$ Riemannian manifold. This structure
permits to identify $T_x M$ with its dual, i.e. the cotangent space $T_x M^*$,
 and thus to define a notion of gradient vector. Indeed, given $f$ in $M$, the gradient of $f$ is denoted by
$\nabla_{_{(\cdot,\cdot)}}\:f$ and is uniquely determined by the following
conditions: \\
\hspace*{.3cm}(g$_1$)  tangency condition:  for all $x\in M$, $
\nabla_{_{(\cdot,\cdot)}}\: f (x) \in T_x M^*\simeq T_xM,$ \\
\hspace*{.3cm}(g$_2$) dualility condition:  for all $x\in M$, $v\in
T_x M$, $
df(x)(v)=(\nabla_{_{(\cdot,\cdot)}}\:f(x),v)_x.$ \\
We refer the reader to \cite{DoC,Lan95} for further details.

Let us return to the minimization problem $(P)$. Since $C$ is
open, we can take $M=C$ 
with the usual identification  $T_xC \simeq \RR^n$ for every $x\in C$. Given
a continuous mapping $H:C\to\Sn$, the
 metric defined by
\begin{equation}\label{E:metric}
\forall x \in C,\; \forall u,v\in  \R^n,\; (u,v)_x^{H}=\langle
H(x)u,v\rangle,
\end{equation}
endows $C$ with a ${\cal C}^0$ Riemannian structure. The corresponding
Riemannian gradient vector field of the objective function $f$
restricted to $C$, which we denote by $\gr f_{|_C}$, is given by
\begin{equation}\label{E:riemann}
\gr f_{|_C}(x)=H(x)^{-1}\nabla f(x).
\end{equation}
Next, take $N=\F=C\cap \A$, which is a smooth submanifold of $C$
with $T_x\F \simeq \A_0$ for each $x\in\F$. Definition
(\ref{E:metric}) induces a metric on $\F $ for which the
gradient of the restriction $f_{|_\F}$ is denoted by $\gr
f_{|_\F}$. Conditions $(g_1)$
and $(g_2)$ imply that for all \( x \in \F\)
\begin{equation}\label{E:gradH}
\gr f_{|_\F} (x)=P_x H(x)^{-1} \nabla f(x),
\end{equation}
where, given $x\in C$, $P_x:\R^n\to \A_0$ is the $(\cdot,\cdot)_x^H$-orthogonal
projection onto the linear subspace $\A_0$. Since  $A$ has full rank,
it is easy to see  that
\begin{equation}\label{E:projH}
P_x=I -H(x)^{-1}A^T(AH(x)^{-1}A^T )^{-1}A,
\end{equation}
and we conclude that for all \( x \in \F\)
\begin{equation}\label{E:gradH-expl}
\gr f_{|_\F}(x)=H(x)^{-1}[I-A^T(AH(x)^{-1}A^T
)^{-1}AH(x)^{-1}]\nabla f(x).
\end{equation}

Given $x\in \F$, the vector $-\gr
f_{|_\F}(x)$ can be interpreted as that
direction in $\A_0$  such that $f$ decreases the most steeply  at $x$ with
respect to the metric $(\cdot,\cdot)_x^{H}$. The {\em steepest descent
method}  for the (local) minimization of $f$ on the Riemannian manifold $\F,
(\cdot,\cdot)^H_x$ consists in finding the solution trajectory $x(t)$ of the
  vector field $-\gr f_{|_\F}$ with initial
condition $x^0\in\F$:
\begin{equation}\label{E:steepestD}
\left\{
\begin{array}{l}
\dot{x}+\gr f_{|_\F}(x)=0,\\
x(0)=x^0\in\F.
\end{array}\right.
\end{equation}



\section{Legendre gradient flows in constrained 
optimization}\label{S:gradientflow}

\subsection{Liapounov functionals, variational inequalities and Hessian
metrics}\label{S:Poincare}

This section is intended to motivate the particular class of
Riemannian metrics that is studied in this paper in view of the
asymptotic convergence of the solution to (\ref{E:steepestD}). 

Let us consider  the minimization problem $(P)$ and assume that
$C$ is endowed  with some Riemannian metric $\me^H_x$ as defined
in (\ref{E:metric}). Recall that \( V:\F\mapsto \R \) is  a {\em
Liapounov functional} for the vector field $-\nabla_H f_{|_\F}$
if $\forall x\in \F$, $(-\gr f_{|_\F}(x), \gr V(x))^H _x\leq 0$.
If $x(t)$ is a solution to  (\ref{E:steepestD}), this implies that $ t\mapsto V(x(t))$ is nonincreasing.  Although \( f_{|_\F} \) is indeed a Liapounov
functional for $-\nabla_H f_{|_\F}$, this does not ensure the
convergence of $x(t)$ (see for instance the counterexample of
Palis-De Melo \cite{PalDem82} in the Euclidean case).

Suppose that the objective function $f$ is convex.  For simplicity, we also
assume that $A=0$ so that $\F=C$. In the framework of convex minimization,
the set of minimizers of \( f \) over $\overline{C}$, denoted by
\( \argmin\: _{\overline{C}}\: f \), is characterized in variational terms
as follows:
\begin{equation}
\label{varEuc}
\begin{array}{lcl}
a\in \argmin\: _{\overline{C}}\: f & \Leftrightarrow  & \forall x\in
\overline{C},\: \langle \nabla f(x),x-a\rangle \geq 0.\\
\end{array}
\end{equation}
 Setting
   \(q_{a}(x)=\frac{1}{2}|x-a|^{2} \)  for all $a\in\argmin\: _{\overline{C}}$,
 one observes that $\nabla
q_a(x)=x-a$ and thus,  by (\ref{varEuc}),   $q_{a}$ is a Liapounov functional for
$-\nabla f$. This key property allows one to establish the
asymptotic convergence as $t\to+\infty$ of the  corresponding
steepest descent trajectories;  see \cite{Bru74} for more details
in a very general non-smooth setting. To use the same kind of
arguments in a non Euclidean context, observe that by
(\ref{E:riemann}) together with the continuity of $ \nabla f$,
the following
variational  Riemannian characterization holds
\begin{equation}
\label{varRie}
\begin{array}{lcl}
a\in \argmin\: _{\overline{C}}\: f & \Leftrightarrow  & \forall x\in C,\: (
\nabla_{H} f(x),x-a)^H _x \geq 0.\\
\end{array}
\end{equation}
We are thus naturally led to the problem of {\it finding the Riemannian 
metrics
on $C$  for which the mappings $C\ni x\mapsto x-y\in\RR^n$, $y\in C$, are
gradient vector fields}. The next result gives a characterization of such metrics: they are induced by Hessian of strictly convex functions.
\begin{theorem}\label{T:Poincare}
Assume that $H\in{\cal
C}^1(C;\Sn)$, or in other words that $\me^H_x$ is a ${\cal C}^1$ metric. The family of vector fields $\{V^y: C\ni x\mapsto x-y\in\RR^n\},\:y\in C$ 
is a family of $\me^H$-gradient vector fields if and only if 
there exists a strictly convex function
\(h\in {\cal C}^3(C) \)  such that \(\forall x\in C\), \( H(x)=\nabla
^{2}h(x) \). Besides, defining
\( D_{h}:C\times C\mapsto \R \) by
\begin{equation}
\label{E:BRE}
D_{h}(y,x)=h(y)-h(x)-\langle \nabla h(x),x-y\rangle,
\end{equation}
we obtain \(
\nabla _{H}D_{h}(y,\cdot )(x)=x-y.\)
\end{theorem}
\begin{proof} The set of metrics complying with the ``gradient''requirement is denoted by \({\cal M} \), that
is, \( \me^H_x\in {\cal M} \Leftrightarrow H\in{\cal
C}^1(C;\Sn)\hbox{ and } \forall y\in C,\: \exists \varphi _{y}\in
{\cal C}^{1}(C;\R),\: \gr \varphi _{y}(x)=x-y.\). Let $(x_1,..,x_n)$ denote the canonical coordinates
of $\R^n$ and  write \(\sum _{i,j}H_{ij}(x)dx_{i}dx_{j} \) for
\(\me^H_x\). By (\ref{E:riemann}), the mappings \( x\mapsto
x-y\), \( y\in C \), define a family of \( \me^H_x\) gradients
iff \( k_{y}:x\mapsto H(x)(x-y)\), \( y\in C \), is a family of
Euclidean  gradients. Setting  \( \alpha ^{y}(x)=\langle
k_{y}(x),\cdot \rangle  \), \( x, y\in C \), the problem amounts
to find necessary (and sufficient) conditions under which the \(
1 \)-forms \( \alpha ^{y} \) are all exact. Let $y\in C$. Since
\( C \) is convex, the Poincar\'e lemma \cite[Theorem V.4.1]{Lan95} states that \( \alpha
^{y} \) is exact iff it is closed. In canonical coordinates we
have \( \alpha ^{y}(x)=\sum_{i} \left(
\sum_{k}H_{ik}(x)(x_{k}-y_{k})\right) dx_{i},\; x\in C,\) and
therefore \( \alpha ^{y} \) is exact iff for all \(i,j\in
\{1,..,n\}\) we have \( \frac{\partial }{\partial
x_{j}}\sum_{k}H_{ik}(x)(x_{k}-y_{k}) =\frac{\partial }{\partial
x_{i}}\sum_{k}H_{jk}(x)(x_{k}-y_{k}),\) which is equivalent to \(
\sum_{k}\frac{\partial }{\partial
x_{j}}H_{ik}(x)(x_{k}-y_{k})+H_{ij}(x)=\sum_{k}\frac{\partial
}{\partial x_{i}}H_{jk}(x)(x_{k}-y_{k})+H_{ji}(x).\) Since \(
H_{ij}(x)=H_{ji}(x) \), this gives the following condition: \(
\sum_{k}\frac{\partial }{\partial
x_{j}}H_{ik}(x)(x_{k}-y_{k})=\sum_{k}\frac{\partial }{\partial
x_{i}}H_{jk}(x)(x_{k}-y_{k}),\:\forall i,j\in \{1,..,n\}.\) If we
set \( V_{x}=(\frac{\partial }{\partial
x_{j}}H_{i1}(x),..,\frac{\partial } {\partial
x_{j}}H_{in}(x))^{T} \) and \( W_{x}=(\frac{\partial }{\partial
x_{i}}H_{j1}(x) ,..,\frac{\partial }{\partial
x_{i}}H_{jn}(x))^{T} \), the latter can be rewritten \( \langle
V_{x}-W_{x},x-y\rangle =0\), which must hold for all $(x,y)\in
C\times C$. Fix $x\in C$. Let \( \epsilon _{x}>0 \) be such that
the open ball of center \( x \) with radius \( \epsilon _{x} \)
is contained in \( C \). For every \( \nu  \) such that \(
|\nu|=1 \), take \( y= x+\epsilon_{x}/2 \nu  \) to obtain that \(
\langle V_{x}-W_{x},\nu \rangle =0\). Consequently, \(
V_{x}=W_{x} \) for all \( x\in C \). Therefore, \(\me^H_x\in{\cal
M}\) iff
\begin{equation}
\label{partialhessian}
\forall x \in C,\:\forall i,j,k \in \{1,..,n\},\:\frac{\partial}
{\partial x_i} H_{jk}(x)=\frac{\partial}
{\partial x_j} H_{ik}(x).
\end{equation}
\begin{lemma}\label{L:Poincare}
If \(H :C\mapsto \Sn \) is a differentiable mapping satisfying {\rm
(\ref{partialhessian})}, then there exists \( h\in {\cal C}^{3}(C) \)   such
that \(\forall x \in C \), \( H (x)=\nabla^2 h(x) \). In particular, $h$ is
strictly convex.
\end{lemma}
\begin{proof}[of Lemma \ref{L:Poincare}.] For all \( i\in \{1,..,n\} \),
set \( \beta ^{i}=\sum _{k}H _{ik}dx_{k} \). By (\ref{partialhessian}), \(
\beta ^{i} \) is closed and therefore exact. Let \( \phi _{i} :C\mapsto \R\)
be such that \( d\phi _{i}=\beta ^{i} \) on \( C \), and set \( \omega =\sum
_{k}\phi _{k}dx_{k} \). We have that \( \frac{\partial }{\partial x_{j}}\phi
_{i}(x)=H_{ij}(x)
=H_{ji}(x)=\frac{\partial }{\partial x_{i}}\phi _{j}(x)\), \(\forall x\in
C\). This proves that \( \omega  \) is closed, and therefore there exists
$ h\in {\cal C}^{2}(C,\R)$ such that \( dh=\omega  \). To conclude we just
have to notice
that \( \frac{\partial }{\partial x_{i}}h(x)=\phi _{i}, \)
and thus \( \frac{\partial^2 h}{\partial x_{j}\partial x_{i}}(x)=H _{ji}(x)
,\:\forall x\in C\). \end{proof} 

To finish the proof, remark that taking $\varphi_y=D_h(y,\cdot)$ with $D_h$
being defined by  (\ref{E:BRE}), we obtain \(\nabla \varphi_y(x)=\nabla^2
h(x)(x-y)\), and therefore \( \nabla _{H}\varphi_y(x)=x-y \) in virtue of
(\ref{E:riemann}). 
\end{proof}
 
\begin{remark}
{\rm (a) In the theory of Bregman proximal methods for convex optimization,
the distance-like function $D_h$ defined  by  {\rm (\ref{E:BRE})} is called
the $D$-{\em function} of  $h$.   Theorem {\rm \ref{T:Poincare}} is a new
 and surprising motivation for the introduction of $D_h$ in relation with variational
inequality problems. (b) For a geometrical approach to Hessian Riemannian structures the reader
is referred to the recent work of Duistermaat {\rm \cite{Dui01}}.}
\end{remark}

Theorem \ref{T:Poincare} suggests to endow $C$ with a Riemannian
structure associated  with the Hessian $H=\nabla^2h$ of a
strictly convex  function $h : C\mapsto \R$. As we will see under
some additional conditions, the $D$-function of $h$ is essential
to establish the asymptotic convergence of the trajectory. On the
other hand, if it is possible to replace $h$ by a sufficiently
smooth strictly convex function $h':C'\mapsto \R$ with
$C'\supset\supset C$ and $h'_{|_C}=h$, then the gradient flows
for $h$ and $h'$ are the same  on $C$ but the  steepest descent
trajectories associated with the latter may leave the feasible
set of $(P)$ and in general they will not converge to a solution
of  $(P)$. We shall see that to avoid this drawback  it is
sufficient  to require that $|\nabla h(x^j)|\rightarrow +\infty$
for all sequences $(x^j)$ in $C$ converging to a boundary point
of  $C$. This may be interpreted as a sort of {\em barrier technique}, a classical strategy to enforce feasibility in
optimization theory.

\subsection{Legendre type functions and the $\HS$ dynamical
system}\label{S:legendre} In the
sequel, we adopt the standard notations of convex analysis
theory; see \cite{Roc70}. Given a closed convex subset $S$ of
$\R^n$, we say that an extended-real-valued function $g:S\mapsto
\R \cup \{+\infty \}$ belongs to the class $\Gamma_0(S)$ when $g$
is lower semicontinuous, proper ($g\not\equiv+\infty$) and
convex. For such a function $g\in\Gamma_0(S)$, its {\em effective
domain}
is defined by $\dom g=\{x\in S\;|\; g(x)<+\infty\}$. When
$g\in\Gamma_0(\R^n)$ its {\em
Legendre-Fenchel conjugate} is given by
$g^*(y)=\sup\{\<x,y\ra-g(x)\;|\;x\in\RR^n\}$, and its {\em
subdifferential}  is the set-valued mapping $\partial g:\RR^n\to
{\cal P}(\RR^n)$ given by $\partial g (x)=\{y\in\RR^n\;|\;
\forall z\in \RR^n, \: f(x)+\<y,z-x\ra\leq f(z)\}$. We set
$\dom\partial g=\{x\in\RR^n\;|\;\partial g(x)\neq\emptyset\}$.

\begin{definition}\label{D:legendre}{\rm  \cite[Chapter
26]{Roc70}} A function $h\in \Gamma_0(\RR^n)$ is
called:\\
{\rm (i)} {\rm essentially smooth}, if $h$ is differentiable on
$\inte\dom h$, with moreover $|\nabla h (x^j)| \rightarrow
+\infty$ for every sequence $(x^j)\subset\inte \dom h$ converging
to a boundary point of $\dom h$ as $j\rightarrow +\infty$;\\
{\rm (ii)}  {\rm of Legendre type}  if $h$ is essentially smooth
and strictly convex on  $\inte\dom h$.
\end{definition}

Remark that by \cite[Theorem 26.1]{Roc70}, $h\in \Gamma_0(\RR^n)$
is essentially smooth iff  $\partial h(x)=\{\nabla h(x)\}$ if
$x\in \inte \dom h$ and $\partial h(x)=\emptyset$ otherwise; in
particular, $\dom\partial h=\inte\dom h$.

Motivated by the results of section \ref{S:Poincare},
we define a Riemannian structure on $C$ by introducing a function
$h\in \Gamma_0(\RR^n)$ such that:
$$
\left\{
\begin{array}{cl}
{\rm (i)}& h\hbox{ is of Legendre type with } \inte\dom h=C.\\
{\rm (ii)}& h_{|_C}\in {\cal C}^2(C;\RR) \hbox{ and } \forall x\in C, \nabla
^2 h(x)\in \Sn.\\
{\rm (iii)}& \hbox{The mapping } C\ni x \mapsto \nabla ^2 h(x)
\hbox{ is locally Lipschitz continuous}.
\end{array}
\right. \leqno (H_0)
$$
Here and subsequently, we take $H=\nabla^2 h$ with $h$  satisfying
$(H_0)$. The Hessian mapping $ C\ni x \mapsto H(x)$ endows $C$
with the (locally Lipschitz continuous) Riemannian metric
\begin{equation}\label{E:metrich}
\forall x \in C,\: \forall u,v\in  \R^n,\: (u,v)_x ^H=\langle
H(x)u,v\rangle=\langle \nabla^2 h(x)u,v\rangle,
\end{equation}
and we say that $(\cdot,\cdot)_x^{H}$ is the {\it Legendre metric}
on $C$ induced by the Legendre type function $h$, which also
defines  a metric on $\F=C\cap \A $ by restriction. In addition
to $f\in{\cal C}^1(\R^n)$, we suppose that the objective function
satisfies
\begin{equation}\label{E:locLip}
\nabla f \hbox{ is locally Lipschitz continuous
on } \R^n.
\end{equation}
The corresponding steepest descent method in the manifold $\F,
(\cdot,\cdot)^H_x$, which we refer to as $\HS$ for short, is then
the following continuous dynamical system
$$  \left\{ \begin{array}{l}
\dot{x}(t)+\gr f_{|_\F}(x(t))=0,\: t\in (T_m, T_M),\\
x(0)=x^0\in\F,
\end{array}\right.
\leqno (H\hbox{-}SD)
$$
 with
$H=\nabla^2 h$ and where $-\infty\leq T_m<0<T_M\leq +\infty$ define the interval corresponding
to the unique  maximal solution of $\HS$.
 Given an initial condition $x^0\in\F$, we
shall say that $(H\hbox{-}SD)$
is  {\em well-posed}  when its maximal solution satisfies
$T_M=+\infty$. In section \ref{S:wellposed} we will give   some  sufficient
conditions ensuring the well-posedness of $\HS$.

\subsection{Differential inclusion formulation of $\HS$ and some
consequences}
It is easily seen that the solution $x(t)$
of $(H\hbox{-}SD)$ satisfies:
\begin{equation}\label{E:cont-prox}
\left\{
\begin{array}{rcl}
\displaystyle{\frac{d}{dt}\nabla h(x(t))+\nabla f(x(t))}&\in&
\A_0^\perp\textrm{ on } (T_m,T_M),\\
\displaystyle{x(t)}&\in& \F\:\: \textrm{ on }
(T_m,T_M),\\x(0)&=&x^0\in\F.
\end{array}
\right.
\end{equation}
This differential inclusion problem  makes sense even when $x\in
W^{1,1}_{loc}(T_m,T_M;\R^n)$,  the inclusions being satisfied
almost everywhere on $(T_m,T_M)$. Actually, the following result
establishes that $\HS$ and (\ref{E:cont-prox}) describe the same
trajectory.
\begin{proposition}\label{difinc}
Let $x\in W^{1,1}_{loc}(T_m,T_M;\R^n)$. Then, $x$ is a solution
of {\rm (\ref{E:cont-prox})} iff $x$ is the solution of
$(H\hbox{-}SD)$. In particular, {\rm (\ref{E:cont-prox})} admits a unique
solution of class ${\cal C}^1$.
\end{proposition}
\begin{proof} Assume that $x$ is a solution of (\ref{E:cont-prox}),
and let $I'$ be the subset of $(T_m,T_M)$ on  which $t\mapsto
(x(t),\nabla h(x(t))$ is derivable. We may assume that $x(t)\in
\F$ and $\frac{d}{dt}\nabla h(x(t))+\nabla f(x(t)) \in \A_0
^{\perp}$, $\forall t\in I'$. Since  $x$ is absolutely
continuous, $\dot x(t)+H(x(t))^{-1}\nabla f((x(t)) \in
H(x(t))^{-1}\A_0 ^{\perp}$ and $\dot x(t)\in \A_0$, $\forall t\in
I'$. But  the orthogonal complement of $\A_0$ with respect to the
inner product $\<H(x)\cdot,\cdot\ra$ is exactly $H(x)^{-1}\A_0
^{\perp}$ when $x\in \F$. It follows that $\dot
x+P_{x}H(x)^{-1}\nabla f(x)=0$ on $I'$. This implies that $x$ is
the ${\cal C}^1$ solution of $(H\hbox{-}SD)$. \end{proof}




  Suppose that $f$ is convex.  On account of Proposition
\ref{difinc}, $(H\hbox{-}SD)$ can be interpreted
as a continuous-time model for  a well-known class of iterative
minimization algorithms. In fact, an implicit discretization of
(\ref{E:cont-prox})  yields the following iterative scheme: $
\nabla h(x^{k+1})-\nabla h (x^k)+\mu_k\nabla f (x^{k+1})\in \Im
A^T,\; Ax^{k+1}=b,$ where $\mu_k>0$ is a step-size parameter and
$x^0\in\F$. This is the optimality condition for
\begin{equation}\label{E:BPM}
x^{k+1}\in \argmin \left\{ f(x)+1/\mu_k D_h(x,x^k)\;|\;
Ax=b\right\},
\end{equation}
where  $D_h$ is given by
\begin{equation}\label{E:Dfunction}
D_h (x,y)=h(x)-h(y)-\langle \nabla h(y), x-y\rangle,\:x\in \dom
h,\;y\in \dom \partial h=C.
\end{equation} 
The above algorithm is accordingly called the {\em Bregman proximal
minimization} method; for an insight of its importance in optimization see  for
instance  \cite{CeZ92,ChT93,IuM00,Kiw97b}.

Next, assume  that  $f(x)=\<c,x\ra$ for some $c\in\R^n$. As
already noticed in \cite{BaL89,Fia90,Mac89} for the log-metric
and in \cite{IuS99} for a  fairly general $h$, in this case the
$\HS$ gradient trajectory can be viewed as a {\em central optimal
path}. Indeed, integrating (\ref{E:cont-prox}) over \( [0,t] \)
we obtain $\nabla h(x(t))-\nabla h(x^0)+tc\in \A_0^\perp$. Since
$x(t)\in \A$, it follows that
\begin{equation}\label{E:viscosity}
x(t)\in \argmin \left\{ \< c,x\ra +1/tD_h(x,x^0)\mid Ax=b\right\},
\end{equation}
which corresponds to  the so-called {\em viscosity method} relative to
$g(x)=D_h(x,x^0)$; see
\cite{Att96,ACH97,IuS99} and Corollary \ref{C:selection}. 
     Remark now that for a linear
objective function, (\ref{E:BPM}) and (\ref{E:viscosity}) are
essentially the same:  the sequence generated by the former
belongs to the optimal path defined by the latter. Indeed,
setting $t_0=0$ and $t_{k+1}=t_k+\mu_k$ for all $k\geq 0$
($\mu_0=0$) and integrating (\ref{E:cont-prox}) over $[t_k,
t_{k+1}]$, we obtain that $x(t_{k+1})$ satisfies the optimality
condition for (\ref{E:BPM}). The following result summarizes the
previous discussion.
\begin{proposition}\label{P:optBreRie}
Assume that  $f$ is linear and that
the corresponding $\HS$ dynamical system is well-posed. Then,
the viscosity optimal path \( \widetilde{x}(\eps) \)  relative to
$g(x)=D_h(x,x^0)$ and the sequence \( (x^{k})\) generated by {\rm
(\ref{E:BPM})}  exist and are unique, with in addition
$\widetilde{x}(\eps)=x(1/\eps)$, $\forall \eps>0$, and
$x^{k}=x(\sum^{k-1}_{l=0}\mu _{l})$, $\forall\ k\geq 1$, where
$x(t)$ is the solution of $\HS$.
\end{proposition}

\begin{remark}
{\rm In order to ensure asymptotic convergence for proximal-type
algorithms, it is usually required that  the step-size parameters
satisfy $\sum \mu_k=+\infty$ . By Proposition {\rm
\ref{P:optBreRie}}, this is necessary for the convergence of {\rm
(\ref{E:BPM})} in the sense that when $\HS$ is well-posed, if
$x^k$ converges to some $x^*\in S(P)$ then either $x^0=x^*$ or
$\sum\mu_k=+\infty$.}
\end{remark}




\section{Global existence, asymptotic analysis and examples}\label{GEAS}




\subsection{Well-posedness of $\HS$}\label{S:wellposed}

  In this section we establish the
well-posedness of $(H\hbox{-}SD)$ (i.e. $T_M=+\infty$)  under three different conditions. In order to
avoid any confusion,  we say that a set $E\subset \RR^n$ is {\em
bounded} when it is so for the usual Euclidean norm
$|y|=\sqrt{\langle y,y\rangle}$. First, we  propose the condition:
$$
\hbox{The lower level set $\{ y\in \overline{\F} \;|\;f(y)\leq
f(x^0)\}$ is bounded.} \leqno (WP_1)
$$
Notice that $(WP_1)$ is weaker than the classical assumption
imposing $f$ to have bounded lower level sets in the $H$ metric
sense. Next, let $D_h$ be the $D$-function of $h$ that is defined
by (\ref{E:Dfunction})   and consider the following condition:
$$
\left\{
\begin{array}{l}
{\rm (i)}\:\:  \dom h=\overline{C} \hbox{  and $\forall a\in
\overline{C}$, $\forall\gamma \in \R$},\: \{y \in \F \:| D_h (a,y)\leq
\gamma\}\hbox{ is bounded}.\\
{\rm (ii)}\:  \hbox{$S(P)\neq\emptyset$  and $f$ is
quasi-convex (i.e. the lower level sets of $f$ are convex)}.
\end{array}
\right. \leqno (WP_2)
$$
When $\overline{\F}$ is unbounded  $(WP_1)$ and $(WP_2)$ involve
some a priori properties on $f$. This is actually not necessary
for the  well-posedness of $\HS$. Consider:
$$
\hbox{$\exists$ $K\geq 0$, $L\in \R$ such that $\forall x\in
C$, } ||H(x)^{-1}||\leq K|x|+L.  \leqno (WP_3)
$$
This property is satisfied by relevant Legendre type functions;
take for instance
(\ref{E:xlogx}).

\begin{theorem}\label{T:existence}
Assume that   {\rm (\ref{E:locLip})} and $(H_0)$
hold and
additionally that either $(WP_1)$,  $(WP_2)$ or $(WP_3)$ is
satisfied. If $\inf _{_{\F}} f>-\infty$ then the dynamical
system $(H$-$SD)$ is well-posed. Consequently,
the mapping $t \mapsto f(x(t))$ is nonincreasing and convergent as
$t\infi$.
\end{theorem}
\begin{proof} When no confusion may occur, we drop the dependence
on  the time variable $t$. By definition, 
$$
T_M=\sup \{T>0\:| \exists !\textrm{ solution $x$ of
$(H$-$SD)$ on $[0,T)$ s.t. } x([0,T)) \subset \F\}.
$$ 
We have that $T_M>0$. The definition
(\ref{E:projH}) of $P_x$ implies that for all $y \in \A_0$,
$(H(x)^{-1}\nabla f(x)+\dot x, y+\dot x)_x^H=0$ on $[0,T_M)$
and therefore
\begin{equation}\label{liap} \langle \nabla f(x)+H(x)\dot x, y+\dot
x\rangle=0
\T{ on } [0,T_M).
\end{equation}
Letting $y=0$ in (\ref{liap}), yields
\begin{equation}\label{fliap}
\frac{d}{dt}f(x)+\langle H(x)\dot x,\dot x\rangle=0.
\end{equation}
By (\ref{E:hypof})(ii), $f(x(t))$ is convergent as $t\to
T_M$. Moreover
\begin{equation}
\label{velo} \langle H(x(\cdot))\dot x(\cdot),\dot
x(\cdot)\rangle \in L^1(0,T_M;\R).
\end{equation}
Suppose that $T_M<+\infty$. To obtain a contradiction, we
begin by proving that $x$ is bounded. If $(WP_1)$ holds then $x$ is bounded
because  $f(x(t))$ is
non-increasing so that $x(t)\in \{ y\in \overline{\F} \:|f(y)\leq
f(x^0)\}$, $\forall t\in [0,T_M)$. Assume now that $f$ and $h$ comply with
$(WP_2)$, and let $a\in
\overline{\F}$.   For each $t\in [0,T_M)$ take $y=x(t)-a$ in (\ref{liap}) to
obtain $
\langle \nabla f(x)+\frac{d}{dt}\nabla h(x),x-a+\dot{x}\rangle =0.$ By
(\ref{fliap}), this gives $ \langle \frac{d}{dt}\nabla
h(x),x-a\rangle+\langle
\nabla f(x), x-a\rangle =0$, which we rewrite as
\begin{equation}
\label{Dliap} \frac{d}{dt} D_h (a,x(t))+\langle \nabla f(x(t)),
x(t)-a\rangle =0,\:\forall t\in [0,T_M).
\end{equation}
Now, let  $a\in\overline{F}$  be a minimizer of $f$ on
$\overline{\F}$. From the
quasi-convexity property of $f$, it follows that $\forall t\in [0,T_M)$,
$\langle \nabla f(x(t)), x(t)-a\rangle \geq 0$. Therefore, $D_h(a,x(t))$ is
non-increasing and $(WP_2)$(ii) implies that $x$  is bounded. Suppose
that $(WP_3)$ holds and fix $ t\in [0,T_M)$, we have
$|x(t)-x^0| \leq \int_0 ^t |\dot x(s)|ds \leq \int_0 ^t
||\sqrt{H(x(s))^{-1}}|||\sqrt{H(x(s))}\:\dot x(s)|ds \leq (\int_0 ^t
||H(x(s))^{-1}||ds)^{1/2}(\int_0 ^t \langle H(x(s))\dot x(s),\dot
x(s)\rangle ds)^{1/2}$. The latter follows from the
Cau\-chy-Schwartz inequality together with  the fact that $\|H(x)\|^2$ is the biggest
eigenvalue of $H(x)$. Thus $
|x(t)-x^0| \leq 1/2[\int_0 ^t
||H(x(s))^{-1}||ds+\int_0 ^t \langle H(x(s))\dot x(s), \dot x(r)\rangle
ds].$ Combining $(WP_3)$ and (\ref{velo}),
Gronwall's lemma yields the boundedness of $x$.

Let $\omega(x^0)$ be the set of limit points of $x$, and set
$K=x([0,T_M))\cup \omega(x^0)$. Since $x$ is bounded,
$\omega(x^0)\neq \emptyset$ and $K$ is compact. If $K\subset C$ then the
compactness of $K$  implies that $x$ can be extended beyond $T_M$,
which contradicts the maximality of $T_M$. Let us prove $K \subset C$. We
argue again by contradiction.
Assume that $x(t_j) \rightarrow x^*$, with $t_j<T_M$, $t_j
\rightarrow T_M$ as $ j \infi$ and  $x^*\in \bd C=
\overline{C}\setminus C$. Since $h$ is of Legendre type, we have
$|\nabla h (x(t_j))| \infi$, and we may assume that  $\nabla h (x(t_j))/
|\nabla h (x(t_j))|\to\nu \in \R^n$ with $|\nu|=1$.
\begin{lemma}\label{L:normal} If $(x^j)\subset C$ is  such that
$x^j\to x^*\in \bd C$ and $\nabla h (x^j)/ |\nabla h (x^j))|\to
\nu \in \R^n$, $h$ being a function of Legendre type with $C=\inte\dom h$,
then $\nu \in
N_{\overline{C}}(x^*)$.
\end{lemma}
\begin{proof}[of Lemma \ref{L:normal}]. By convexity of $h$,
$\langle \nabla h (x^j)-\nabla h(y), x^j-y\rangle \geq 0$ for all
$y\in C$. Dividing by $|\nabla h (x^j)|$ and letting $j \infi$,
we get $\<\nu,y-x^*\ra\leq 0$ for all $y\in C$, which holds also
for $y\in\overline{C}$. Hence, $\nu \in N_{\overline{C}}(x^*)$.
\end{proof} 

Therefore, $\nu \in N_{\overline{C}}(x^*)$. Let $\nu
_0=\Pi_{\A_0}\nu$ be the Euclidean orthogonal
projection of $\nu$ onto $\A_0$, and take $y=\nu_0$ in (\ref{liap}).
  Using (\ref{fliap}), integration  gives
\begin{equation}
\label{blowup} \langle \nabla h (x(t_j)),\nu _0\rangle=\langle
\nabla h (x^0)-\int_0 ^{t_j} \nabla f(x(s))ds,\nu _0\rangle.
\end{equation}
By $(H_0)$ and the boundedness property of $x$, the right-hand
side of (\ref{blowup}) is bounded under the assumption $T_M<+\infty$. Hence,
to draw a contradiction from
  (\ref{blowup}) it suffices to prove
$\langle \nabla h (x(t_j)), \nu_0\rangle \infi$. Since $\langle
\nabla h (x(t_j))/ |\nabla h (x(t_j))|, \nu _0\rangle \rightarrow
|\nu _0|^2$, the proof of the result is complete if we check that
$\nu _0 \neq 0$. This is a direct consequence of the following
\begin{lemma}\label{L:lem1}
Let $C$ be a nonempty open convex subset of $\R^n$ and  ${\cal
A}$  an affine subspace of $\R^n$ such that $C \cap {\cal A} \neq
\emptyset.$ If $x^* \in (\bd C )\cap {\cal A}$ then $
N_{\overline{C}} (x^*)\cap {\cal A}_0^{\perp}=\{0\}$ with
$\A_0=\A-\A$.
\end{lemma}
\begin{proof}[of Lemma \ref{L:lem1}]. Let us argue by contradiction
and suppose that we can pick some $v\neq 0$ in ${\cal A}_0
^{\perp}\cap N_{\overline{C}} (x^*)$. For $y_0 \in C \cap {\cal A}
$ we have $\langle v, x^*-y_0 \rangle=0$. For $r\geq 0$, $z\in
\R^n$,
let $B(z,r)$ denote the ball with center $z$ and radius $r$. There exists
$\epsilon >0$, such that
$B(y_0,\epsilon) \subset C$. Take $w$ in $B(0,\epsilon )$ such that $\langle
v,w\rangle <0$, then
  $y_0+w \in C$, yet $\langle v, x^*-(y_0+w) \rangle=\langle v,w\rangle <0$.
This contradicts the fact
that $v$ is in $N_{\overline{C}} (x^*)$. \end{proof}

This completes the proof of the theorem.
\end{proof}
\subsection{Value convergence for a convex objective 
function}\label{S:value}
As a first result concerning the asymptotic behavior of $\HS$, we have the
following:
\begin{proposition}\label{P:convergence}
If $\HS$  is well-posed  and $f$ is convex then \( \forall
a\in\F,\;\forall t>0,\; f(x(t))\leq f(a)+\frac{1}{t}D_h(a,x^0)\),
where $D_h$ is defined by {\rm (\ref{E:Dfunction})}, hence $ \lim\limits_{t\to
+\infty}f(x(t))=\inf_{\overline{\F}}f.$
\end{proposition}
\begin{proof} We begin by noticing that  $f(x(t))$ converges as
$t\to+\infty$ (see Theorem \ref{T:existence}). Fix $a\in \F$. By
(\ref{Dliap}), we have that the solution $x(t)$ of $(H$-$SD)$
satisfies $ \frac{d}{dt} D_h(a,x(t))+\langle \nabla
f(x(t)),x(t)-a\rangle=0, \: \forall t\geq 0.$ The convex
inequality$ f(x)+\<\nabla f(x),x-a\ra\leq f(a)$ yields $
D_h(a,x(t))+\int_0^t[f(x(s))-f(a)]ds\leq D_h(a,x^0).$ Using that
$D_h\geq 0$ and since $f(x(t))$ is non-increasing, we get the
estimate. Letting $t\to+\infty$, it follows that $ \lim_{t\to
+\infty}f(x(t))\leq f(a). $ Since $a\in\F$ was arbitrary chosen,
the proof  is complete. \end{proof}

%




\subsection{Bregman metrics and trajectory convergence}\label{S:bregman}

  In this section we establish the convergence of $x(t)$ under some
additional properties on the
$D$-function of $h$. Let us begin with a definition.
\begin{definition}\label{D:bregman} A function $h\in\Gamma_0(\RR^n)$ is
called  {\rm Bregman function with zone $C$} when the following
conditions are satisfied:\\
{\rm (i)} $\dom h=\overline{C}$,  $h$ is continuous and strictly convex on
$\overline{C}$ and $h_{|_C}\in {\cal C}^1(C;\RR)$.\\
{\rm (ii)}  $\forall a\in \overline{C}$, $\forall \gamma\in\RR$,
$\{y\in C| D_h(a,y)\leq \gamma\}$ is bounded, where $D_h$ is defined by {\rm
(\ref{E:Dfunction})}.\\
{\rm (iii)} $\forall y\in\overline{C}$, $\forall y^j\to y$ with
$y^j\in C$, $D_h(y,y^j)\to 0$.
\end{definition}

Observe that this notion slightly weakens the usual definition of
Bregman function that was proposed by Censor and Lent in
\cite{CeL81}; see also \cite{Bre67}.  Actually, a Bregman function
in the sense of Definition \ref{D:bregman} belongs to the class of
$B$-functions introduced by Kiwiel (see \cite[Definition
2.4]{Kiw97a}). Recall the following  important asymptotic
separation property:
\begin{lemma}\label{L:Bregman}{\rm \cite[Lemma
2.16]{Kiw97a}} If $h$ is a Bregman function with zone $C$ then
$\forall y\in\overline{C}$, $\forall (y^j)\subset C$ such that
$D_h(y,y^j)\to 0$, we have $y^j\to y$.
\end{lemma}

\begin{theorem}\label{T:convergence}  Suppose that $(H_0)$  holds with $h$ being
a Bregman function with zone
$C$.  If $f$ is quasi-convex satisfying 
{\rm (\ref{E:locLip})} and $S(P)\neq\emptyset$ then $(H$-$SD)$ is
well-posed  and its solution $x(t)$  converges as $t\to+\infty$
to some $x^*\in \overline{\F}$ with $ -\nabla f (x^*) \in
N_{\overline{C}}(x^*)+ \A_0^\perp. $ If in addition $f$ is convex
then $x(t)$ converges to a solution of $(P)$.
\end{theorem}
\begin{proof} Notice first that $(WP_2)$ is satisfied. By Theorem
\ref{T:existence}, $(H$-$SD)$ is well-posed, $x(t)$ is bounded
and for each $a\in S(P)$, $D_h(a,x(t))$ is non-increasing and
hence convergent 
Set
$f_\infty=\lim_{t\to+\infty}f(x(t)) $ and define $L=\{y\in
\overline{\F}\;|\;f(y)\leq f_\infty\}$. The set $L$ is nonempty
and closed. Since $f$ is supposed to be quasi-convex, $L$ is
convex, and similar arguments as in the proof of Theorem
\ref{T:existence} under $(WP_2)$ show that $D_h(a,x(t))$ is
convergent for all $a\in L$.  Let $x^*\in L$ denote a cluster
point of $x(t)$ and take $t_j\to+\infty$ such that $x(t_j)\to
x^*$. Then, by (iii) in Definition \ref{D:bregman}, $
\lim_{t}D_h(x^*,x(t))=\lim_{j}D_h(x^*,x(t_j))=0.$
Therefore, $x(t)\to x^*$ thanks to
Lemma \ref{L:Bregman}. Let us prove that $x^*$ satisfies the optimality
condition $-\nabla f(x^*)\in
N_{\overline{C}}(x^*)+ \A_0^\perp$. Fix $z\in \A_0$, and for each
$t\geq 0$ take $y=-\dot{x}(t)+z$ in (\ref{liap}) to obtain $
\langle\frac{d}{dt} \nabla h(x(t))+ \nabla f(x(t)),z\rangle=0.$
This gives
\begin{equation}\label{E:inte}
\frac{1}{t}\int_0^t\langle\nabla f(x(s)),z\rangle
ds=\langle s(t),z\rangle,
\end{equation}
where $s(t)=[\nabla h(x^0)-\nabla h(x(t))]/t.$ If $x^*\in \F$ then $\nabla
h(x(t))\to \nabla h(x^*)$, hence $
\langle \nabla f(x^*),z\rangle=\lim_{t\to+\infty}\frac{1}{t}
\int_0^t\langle\nabla f(x(s)),z\rangle ds =\lim_{t\to+\infty}\langle
s(t),z\rangle=0.$ Therefore, $\Pi_{\A_0}\nabla f(x^*)=0$. But
$N_{\overline{\F}}(x^*)=\A_0^\perp$ when $x^*\in \F$, which proves
our claim in this case. Assume now that $x^*\notin\F$, which
implies that $x^*\in\partial C\cap \A$. By (\ref{E:inte}), we
have that $\langle s(t),z\rangle$ converges to $\langle \nabla
f(x^*),z\rangle$ as $t\to+\infty$ for all $z\in \A_0$, and
therefore  $\Pi_{\A_0} s(t)\to\Pi_{\A_0} \nabla f(x^*)$ as
$t\to+\infty$. On the other hand, by Lemma
\ref{L:normal}, we have that there exists $\nu\in
-N_{\overline{C}}(x^*)$ with $|\nu|=1$ such that $ \nabla
h(x(t_j))/|\nabla h(x(t_j))|\to \nu$ for some $t_j\to+\infty$.
Since $N_{\overline{C}}(x^*)$ is positively homogeneous, we
deduce that $\exists$ $\bar\nu\in -N_{\overline{C}}(x^*)$ such
that $ \Pi_{\A_0} \nabla f(x^*)= \Pi_{\A_0} \bar\nu$. Thus,
$-\nabla f(x^*)\in -\Pi_{\A_0}\bar\nu+\A_0^\perp\subseteq
N_{\overline{C}}(x^*)+\A_0^\perp$, which proves the theorem.
\end{proof}

 Following \cite{IuS99}, we remark that when $f$ is
linear, the limit point can be characterized as a sort of
``$D_h$-projection'' of the initial condition onto the optimal
set $S(P)$. In fact, we have:
\begin{corollary}\label{C:selection} Under the assumptions of Theorem {\em
\ref{T:convergence}}, if $f$ is linear then
the solution $x(t)$ of $(H$-$SD)$  converges as $t\to+\infty$ to
the unique optimal solution $x^*$ of
\begin{equation}\label{E:selection}
\min_{x\in S(P)} D_h(x,x^0).
\end{equation}
\end{corollary}
\begin{proof} Let $x^*\in S(P)$ be such that $x(t)\to x^*$ as $t\to +\infty$.
Let $\bar x\in S(P)$.  Since $x(t)\in \F$, the optimality of $\bar x$ yields
$f(x(t))\geq f(\bar x)$, and
it follows from (\ref{E:viscosity}) that $D_h(x(t),x^0)\leq
D_h(\bar x, x^0)$. Letting $t\to+\infty$ in the last inequality, we deduce
that $x^*$ solves
(\ref{E:selection}). Noticing that $D_h(\cdot,x^0)$ is strictly
convex  due to Definition \ref{D:bregman}(i), we conclude the
result. \end{proof}

 We finish this section with an abstract result
concerning the rate of convergence under uniqueness of the
optimal solution. We will apply this result in the next section.
Suppose that $f$ is convex and satisfies (\ref{E:hypof}) and {\rm
(\ref{E:locLip})}, with in addition $S(P)=\{a\}$. Given a Bregman
function $h$ complying with $(H_0)$,  consider the following
growth condition:
$$
  f(x)-f(a)\geq \alpha D_h (a,x)^{\beta}, \: \forall x \in U_a
\cap  \overline{C}, \leqno (GC)
$$
where $U_a$ is a neighborhood of $a$ and with $\alpha>0$, $\beta
\geq 1$.
The next abstract result gives an
estimation of the convergence rate with respect to the
$D$-function of $h$.
\begin{proposition}\label{P:rate1}
Assume that $f$ and $h$ satisfy the above conditions an let
$x:[0,+\infty)\to \F$ be the solution of $\HS$. Then we have the
following estimations:\\
$\bullet$ If $\beta =1$ then there exists $K>0$ such that
$D_h (a,x(t))\leq Ke^{-\alpha t}$, $\forall t>0.$\\
$\bullet$ If $\beta >1$ then there exists $K'>0$ such that $D_h
(a,x(t))\leq K'/t^{\frac{1}{\beta-1}}$, $\forall t>0.$
\end{proposition}
\begin{proof} The assumptions of Theorem \ref{T:convergence} are
satisfied, this yields the well-posedness of $(H$-$SD)$
and the convergence of $x(t)$ to $a$ as $t\to +\infty$.
Besides, from  (\ref{Dliap}) it follows that for all $t\geq 0$,
$\frac{d}{dt} D_h(a,x(t))+
\langle \nabla f (x(t)), x(t)-a \rangle=0.$ By convexity of $f$,
we have $ \frac{d}{dt} D_h(a,x(t))+f (x(t))-f(a)\leq 0.$ Since
$x(t)\to a$, there exists $t_0$
such that $\forall t\geq t_0$,  $x(t)\in U_a \cap \F$. Therefore
by combining $(GC)$ and the last inequality it follows that
\begin{equation}
\label{difine} \frac{d}{dt} D_h(a,x(t))+\alpha D_h
(a,x(t))^{\beta}\leq 0, \: \forall t\geq t_0.
\end{equation}
In order to integrate this differential inequality, let us first
observe that we have the following equivalence: $D_h(a,x(t))>0,
\:\forall t\geq0$ iff $x^0\neq a$. Indeed, if $a \in
\overline{\F}\setminus\F$ then the equivalence follows from
$x(t)\in\F$ together with Lemma \ref{L:Bregman}; if $a\in \F$ then
the optimality condition that is satisfied by $a$ is
$\Pi_{\A_0}\nabla f(a)=0$, and the equivalence
is a consequence of the uniqueness of the solution $x(t)$ of $(H$-$SD)$.
Hence, we can assume that $x^0\neq a$ and
divide (\ref{difine}) by $D_h (a,x(t))^{\beta}$ for all $t\geq
t_0$.
A simple integration procedure then yields the result. \end{proof}

\subsection{Examples: interior point flows in convex
programming}\label{S:Examples}

This section gives a systematic method to construct explicit Legendre metrics 
 on a quite general class of convex sets. By so doing, we will also show
 that many systems studied earlier by various authors \cite{BaL89,Kar90,Fay91a,Fia90,Mac89} appears as particular cases of $\HS$ systems.

Let $p\geq 1$ be an integer and set $I=\{1,\ldots,p\}$. Let us
assume that to each $i\in I$ there corresponds a ${\cal C}^3$ concave
function
$g_i:\RR^n\to \RR$ such that
\begin{equation}\label{E:slater}
\exists x^0\in \RR^n, \; \forall i\in I, \; g_i(x^0)> 0.
\end{equation}
Suppose that the open convex set $C$ is given by
\begin{equation}\label{E:convexC}
C=\{x\in\reels^n\;|\;g_i(x)>0, i\in I\}.
\end{equation}
By  (\ref{E:slater}) we have that
$C\neq\emptyset$ and $\overline{C}=\{x\in\reels^n\;|\;g_i(x)\geq 0, i\in
I\}.$
Let us introduce a class of convex functions of Legendre type $\theta \in
\Gamma_0(\RR)$
satisfying
$$
\left\{
\begin{array}{l}
\hbox{(i)   }(0,\infty)\subset \dom \theta \subset [0,\infty).\\
\hbox{(ii)  }\theta\in {\cal C}^3(0,\infty) \hbox{ and } \lim_{s\to
0^{+}}\theta'(s)=-\infty.\\
\hbox{(iii) }\forall s>0,\: \theta''(s)>0.\\
\hbox{(iv) }\hbox{Either }\theta\hbox{ is non-increasing or }\forall i\in I,
g_i \hbox{ is an affine function.}
\end{array}
\right.\leqno{(H_1)}$$
\begin{proposition} Under {\rm (\ref{E:slater})} and
$(H_1)$, the function $h\in \Gamma_0(\RR^n)$ defined by
\begin{equation}\label{E:defofh} h(x)=\sum_{i\in I}\theta(g_i(x)).
\end{equation} is essentially smooth  with $\inte\dom h=C$ and $h\in {\cal
C}^3(C)$,
where $C$ is given by {\rm (\ref{E:convexC})}. If we assume in
addition the following non-degeneracy condition:
\begin{equation}\label{E:nondeg}
\forall x\in C,\; {\rm span}\{\nabla g_i(x)\:|\:i\in I \}=\reels^n,
\end{equation}
then $H=\nabla^2 h$ is positive definite on $C$, and consequently $h$
satisfies $(H_0)$.
\end{proposition}
\begin{proof} Define $h_i\in\Gamma_0(\RR^n)$ by
$h_i(x)=\theta(g_i(x)).$ We have that $\forall i\in I$,
$C\subset \dom h_i$. Hence $ \inte\dom h=\bigcap_{i\in
I}\inte\dom h_i\supseteq C\neq \emptyset,$ and by \cite[Theorem
23.8]{Roc70}, we conclude that $
\partial h(x)=\sum_{i\in I}\partial h_i(x)$ for all
$x\in\RR^n$. But $ \partial h_i(x)=\theta'(g_i(x))\nabla g_i(x)$
if  $g_i(x)>0$ and  $\partial h_i(x)=\emptyset$ if $g_i(x)\leq
0$; see \cite[Theorem IX.3.6.1]{HiL96}. Therefore
$\partial h(x)=\sum_{i\in I}\theta'(g_i(x))\nabla g_i(x)$ if $x\in C$, and
$\partial h(x)=\emptyset$ otherwise. Since $\partial h$ is a
single-valued mapping, it follows from \cite[Theorem 26.1]
{Roc70} that $h$ is essentially smooth and $\inte\dom
h=\dom\partial h=C$. Clearly, $h$ is of class ${\cal C}^3$ on
$C$. Assume now that (\ref{E:nondeg}) holds. For  $x\in C$, we
have $\nabla^2 h(x)=\sum_{i\in I}\theta''(g_i(x))\nabla
g_i(x)\nabla g_i(x)^T+\sum_{i\in
I}\theta'(g_i(x))\nabla^2g_i(x).$ By $(H_1)$(iv), it follows that
for any $v\in\RR^n$, $ \sum_{i\in
I}\theta'(g_i(x))\<\nabla^2g_i(x)v,v\ra\geq 0$. Let $v\in\RR^n$
be such that $\<\nabla ^2h(x)v,v\ra=0$, which yields $ \sum_{i\in
I}\theta''(g_i(x))\<v,\nabla g_i(x)\ra^2=0.$ According to
$(H_1)$(iii),  the latter implies that $v\in {\rm span}\{\nabla
g_i(x)|i\in I \}^\perp=\{0\}$. Hence $\nabla ^2h(x)\in\Sn$  and
the proof is complete. \end{proof}

  If $h$ is defined by (\ref{E:defofh}) with
$\theta\in\Gamma_0(\RR)$ satisfying $(H_1)$, we say that $\theta$
is the {\em Legendre kernel} of $h$. Such kernels can be divided
into two classes. The first one corresponds to those kernels
$\theta$ for which $\dom\theta=(0,\infty)$ so that
$\theta(0)=+\infty$, and are associated with {\em interior
barrier} methods in optimization as  for instance :  the log-barrier $\theta_1(s)=-\ln(s)$, $s>0$ and the inverse barrier $\theta_{2}(s)=1/s$,
$s>0$. The kernels $\theta$ belonging to the second class
satisfy $\theta(0)<+ \infty$, and are connected with the notion
of {\em Bregman function} in proximal algorithms theory. 
Here are some examples: 
 the Boltzmann-Shannon entropy 
$\theta_3(s)=s\ln(s)-s$, $s\geq 0$ (with
$0\ln0=0$); $\theta_4(s)=-\frac{1}{\gamma}s^\gamma$ with $\gamma\in (0,1)$ ,
$s\geq 0$ (Kiwiel \cite{Kiw97a});  $\theta_5(s)= (\gamma
s-s^{\gamma})/(1-\gamma)$ with $\gamma \in (0,1)$, $s\geq 0$ (Teboulle \cite{Teb92}); the ``$x \log x$'' entropy 
$\theta_6 (s)=s\ln s$, $ s\geq 0$.
In relation with  Theorem \ref{T:convergence} given in the previous section, note that the
Legendre kernels $\theta_i$, $i=3,...,6$, are all Bregman functions with
zone $\RR_+$. Moreover, it is easily seen that each corresponding Legendre 
function $h$ defined by {\rm (\ref{E:defofh})} is indeed 
a Bregman function with zone $C$.

In order to illustrate the type of dynamical systems given by
$(H$-$SD)$,  consider the case of positivity constraints where
$p=n$ and $g_i(x)=x_i$, $i\in I$. Thus $C=\RR^n_{++}$ and
$\overline{C}=\RR^n_+$. Let us assume that $\exists x^0\in
\R^n_{++}$, $ Ax^0=b$. Recall that the corresponding minimization problem is
 $(P)\:\:\min\{ f(x) \;|\; x\geq 0,\; Ax=b\}$ and take first the kernel $\theta_3$ from above. The
associated Legendre function (\ref{E:defofh}) is given by
\begin{equation}\label{E:xlogx}
h(x)=\sum_{i=1}^n x_i\ln x_i-x_i, \; x\in\RR^n_{+},
\end{equation}
and the differential equation in $(H$-$SD)$ is given by
\begin{equation}\label{E:faybusovich}
\dot{x}+[I-XA^T(AXA^T)^{-1}A]X\nabla f(x)=0.
\end{equation}
where $X={\rm diag}(x_1,...,x_n)$.
If $f(x)=\<c,x\ra$ for some $c\in \RR^n$ and  in absence of
linear equality constraints, then (\ref{E:faybusovich}) is $
\dot{x}+X c=0$. The  change of coordinates $y=\nabla h(x)=(\ln
x_1,...,\ln x_n)$ gives $ \dot{y}+c=0$. Hence,
$x(t)=(x_1^0e^{-c_1t},...,x_n^0e^{-c_nt})$, $t\in \R,$ where
$x^0=(x_1^0,...,x_n^0)\in\RR^n_{++}$. If $c\in\R^n_+$ then
$\inf_{x\in \RR^n_{+}}\<c,x\ra=0$ and $x(t)$ converges to a
minimizer of $f=\<c,\cdot\ra$ on $\RR^n_+$; if $c_{i_0}<0$ for
some $i_0$,  then $\inf_{x\in \RR^n_{+}}\<c,x\ra=-\infty$ and
$x_{i_0}(t)\to+\infty$ as $t\rightarrow +\infty$. Next, take
$A=(1,\ldots,1)\in\RR^{1\times n}$ and $b=1$ so that the feasible
set of $(P)$  is given by
$\overline{\F}=\Delta_{n-1}=\{x\in\R^n\mid x\geq 0,\:
\sum_{i=1}^nx_i=1\}$, that is the $(n-1)$-dimensional simplex. In
this case, (\ref{E:faybusovich}) corresponds to $\dot
x+[X-xx^T]\nabla f(x)=0$, or componentwise
\begin{equation}\label{E:karmarkar}
\dot x_i+x_i\left(\frac{\partial f}{\partial
x_i}-\sum_{j=1}^nx_j\frac{\partial f}{\partial x_j}\right)=0,\quad
i=1,\ldots,n.
\end{equation}
For  suitable choices of $f$, this is a {\em Lotka-Volterra} type
equation that naturally arises in population dynamics theory and,
in that context, the structure $(\cdot,\cdot)^H$ with $h$ as in (\ref{E:xlogx}) is usually referred to as
the {\em Shahshahani} metric; see \cite{Aki79,HoS98} and the
references therein. The figure \ref{3dess1} gives
a numerical illustration of system (\ref{E:karmarkar})
for $n=3$ and with $f(x)=x_3-x_2$.
\begin{figure}
\begin{center}
\includegraphics[scale=0.4]{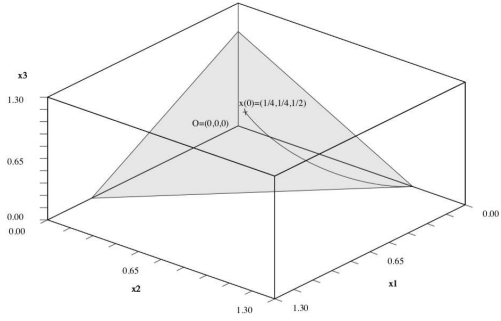}
\end{center}
\caption{\label{3dess1} A trajectory of (\ref{E:karmarkar}).}
\end{figure}
Karmarkar studied (\ref{E:karmarkar}) in \cite{Kar90} for
a quadratic objective function as a continuous model of the
interior point algorithm introduced by him in \cite{Kar84}. 
Equation (\ref{E:faybusovich}) is studied by Faybusovich in
\cite{Fay91a,Fay91b,Fay91c} when $(P)$ is a linear
program, establishing connections with completely integrable
Hamiltonian systems and exponential convergence rate, and by
Herzel et al. in \cite{HRZ91}, who prove quadratic convergence
for an explicit discretization.  

Take now the log barrier kernel $\theta_1$ and $
h(x)=-\sum_{i=1}^n\ln x_i.$
Since $\nabla ^2 h(x)=X^{-2}$ with $X$ defined as above,
the associated differential equation is
\begin{equation}\label{E:Atrajectory}
\dot{x}+[I-X^2A^T(AX^2A^T)^{-1}A]X^2\nabla f(x)=0.
\end{equation}
This equation  was considered by Bayer and Lagarias in
\cite{BaL89} for  a  linear program. In the particular case
$f(x)=\<c,x\ra$ and without linear equality constraints,
(\ref{E:Atrajectory}) amounts to  $\dot{x}+X^2c=0$, or $\dot y+c=0$ for
$y=\nabla h(x)=-X^{-1}e$ with $e=(1,\cdots,1)\in\R^n$, which gives
$x(t)=\left(1/(1/x_1^0+c_1t),...,1/(1/x_n^0+c_nt)\right)$,
$ T_m\leq t\leq T_M, $ with $ T_m=\max\{-1/x_i^0c_i\;|\; c_i>0\}$
and $T_M=\min\{-1/x_i^0c_i\;|\; c_i<0\}$ (see \cite[pag.
515]{BaL89}).   Denote by $\Pi_{A_0}$ the Euclidean orthogonal projection 
onto $\A_0$. To study  the associated trajectories for a general linear
program, it is
introduced in \cite{BaL89} the {\em Legendre transform
coordinates} $y=\Pi_{\A_0}\nabla
h(x)=[I-A^T(AA^T)^{-1}A]X^{-1}e$, which still linearizes
(\ref{E:Atrajectory}) when
$f$ is linear (see  section \ref{S:legendre-transform} for an
extension of this result), and permits to establish some
remarkable analytic and geometric properties of the trajectories.
A similar system was considered in \cite{Fia90,Mac89} as a continuous
log-barrier method for nonlinear inequality constraints and with
$\A_0=\R^n$.

New systems may be derived by choosing other kernels. For instance, taking 
 $h(x)=-1/\gamma\sum_{i=1}^n x_i^\gamma$ with $\gamma\in (0,1)$, $A=(1,\ldots,1)\in\RR^{1\times n}$ and $b=1$, we obtain
\begin{equation}\label{E:kiwiel}
\dot x_i+\frac{x_i^{2-\gamma}}{1-\gamma}\left(\frac{\partial f}{\partial
x_i}-\sum_{j=1}^n\frac{x_j^{2-\gamma}}{\sum_{k=1}^nx_k^{2-\gamma}}\frac{\partial f}{\partial x_j}\right)=0,\quad
i=1,\ldots,n.
\end{equation} 




\subsection{Convergence results for linear programming}\label{S:LP}

Let us consider the specific case of a  linear program
$$
\min_{x\in \reels^n}\{\langle c,x\rangle\mid Bx\geq d,\: Ax= b\},
\leqno{(LP)}
$$
where $A$ and $b$ are as in section \ref{S:problem}, $c\in
\reels^n$,  $B$ is a $p\times n$ full rank real matrix with
$p\geq n$ and $d\in\R^p$. We assume that the optimal set satisfies
\begin{equation}\label{E:LP}
S(LP)\hbox{ is nonempty and bounded},
\end{equation}
and there exists a  Slater point $x^0\in\R^n$, $Bx^0>d$ and
$Ax^0=b$. Take the
Legendre function
\begin{equation}\label{E:defhLP}
h(x)=\sum_{i=1}^n\theta(g_i(x)),\quad g_i(x)= \<B_i,x\ra-d_i,
\end{equation}
where $B_i\in\RR^n$ is the $i$th-row of $B$  and the Legendre
kernel $\theta$ satisfies $(H_1)$.  By (\ref{E:LP}),  $(WP_1)$ holds and
therefore $\HS$ is well-posed due
to Theorem \ref{T:existence}. Moreover,  $x(t)$  is bounded
and all its cluster points belong
to $S(LP)$ by Proposition \ref{P:convergence}.
The variational property
(\ref{E:viscosity}) ensures the convergence of $x(t)$ and gives a
variational characterization of the limit as well. Indeed, we have the
following result:
\begin{proposition}\label{P:selectionLP}
Let $h$ be given by {\rm (\ref{E:defhLP})} with $\theta$
satisfying $(H_1)$. Under {\rm (\ref{E:LP})}, $\HS$ is well-posed
and $x(t)$ converges as $t\to+\infty$ to the
unique  solution $x^*$ of
\begin{equation}\label{E:selectionLP}
\min_{x\in S(LP)}\sum\limits_{i\notin I_0}
D_{\theta}(g_i(x),g_i(x^0)),
\end{equation}
where $I_0=\{i\in I\mid g_i(x)=0\mbox{ for all } x\in S(LP)\}$.
\end{proposition}
\begin{proof}  Assume that $S(LP)$ is not a
singleton, otherwise there is nothing to prove. The
relative interior $\ri S(LP)$ is nonempty and moreover $\ri
S(LP)=\{x\in\reels^n\mid
g_i(x)=0\mbox{ for }i\in I_0,\: g_i(x)>
0\mbox{ for }i\not\in I_0,\: Ax=b\}$. By compactness of $S(LP)$ and strict
convexity of $\theta\circ g_i$, there exists a unique solution
$x^*$ of (\ref{E:selectionLP}). Indeed, it is easy to see that $x^*\in
\ri(LP)$.  Let $\bar x\in S(LP)$ and $t_j\to+\infty$ be
such that $x(t_j)\to\bar x$. It suffices to prove that $\bar x= x^*$. When
$\theta(0)<+\infty$, the latter follows by the same arguments as in
Corollary \ref{C:selection}. When  $\theta(0)=+\infty$,  the proof of
\cite[Theorem 3.1]{ACH97} can be adapted to our setting (see also
\cite[Theorem 2]{IuS99}). Set $x^*(t)=x(t)-\bar x+x^*$. Since
$Ax^*(t)=b$ and $D_h(x,x^0)=\sum_{i=1}^m
D_{\theta}(g_i(x),g_i(x^0))$, (\ref{E:viscosity}) gives
\begin{equation}\label{E:visco}
\langle c,x(t)\rangle+\frac{1}{t}\sum_{i=1}^m
D_{\theta}(g_i(x(t)),g_i(x^0))\leq\langle
c,x^*(t)\rangle+\frac{1}{t}\sum_{i=1}^m
D_{\theta}(g_i(x^*(t)),g_i(x^0)).
\end{equation}
But $\langle c,x(t)\rangle=\langle c,x^*(t)\rangle$ and $\forall
i\in I_0$, $g_i(x^*(t))=g_i(x(t))>0$. Since $x^*\in\ri S(LP)$, for
all $i\notin I_0$ and $j$ large enough, $g_i(x^*(t_j))> 0$. Thus,
the right-hand side of (\ref{E:visco}) is finite at $t_j$, and it
follows that $\sum\limits_{i\notin I_0}D_{\theta}(g_i(\bar
x),g_i(x^0))\leq\sum\limits_{i\notin I_0}
D_{\theta}(g_i(x^*),g_i(x^0)).$ Hence, $\bar x= x^*$. \end{proof}

{\bf Rate of convergence.} We turn now to the case where there is no equality
constraint so that the linear program is
\begin{equation}\label{E:lp}
\min_{x\in\RR^n}\{\<c,x\ra\mid Bx\geq d\}.
\end{equation}
We assume
that (\ref{E:lp}) admits a unique solution $a$ and we study  the
rate of convergence  when $\theta$ is a Bregman function with
zone $\R_{+}$. To apply Proposition \ref{P:rate1}, we need:
\begin{lemma}\label{L:norm}
Set $C=\{x\in\R^n|Bx> d\}$. If {\rm (\ref{E:lp})} admits a unique
solution $a\in\R^n$ then  $\exists k_0>0$,  $\forall y\in
\overline{C}$, $\langle c,y-a\rangle \geq k_0\N(y-a)$, where $\N
(x)=\sum _{i\in I}  |\langle B_i , x\rangle|$ is  a norm on
$\R^n$.
\end{lemma}
\begin{proof} Set $ I_0=\{i\in I\mid \<B_i,a\ra=d_i\}$.  The optimality
conditions for $a$  imply the existence  of a {\em multiplier}
vector
$\lambda \in \R^p_+$ such that $\lambda _i [d_i-\< B_i,a\ra]=0$,
$\forall i\in I,$ and $c=\sum_{i\in I} \lambda_i B_i$. Let
$y\in \overline{C}$. We deduce that $ \< c,y-a \ra=N(y-a)$ where
$N(x)=\sum_{i\in I_0} \lambda_i |\< B_i,x\ra|$. By uniqueness of the optimal
solution, it is easy to see that
${\rm span}\{B_i\mid i\in I_0 \}=\reels^n$, hence $N$ is a norm
on $\R^n$. Since $\N(x)=\sum _{i\in I}  |\langle B_i , x\rangle|$ is also  a
norm on $\R^n$ (recall that $B$ is a full rank matrix), we deduce that
$\exists k_0$ such that $N(x)\geq k_0\N(x)$. \end{proof}

  The following lemma is a sharper version of Proposition
\ref{P:rate1} in the linear
context.
\begin{lemma}
Under the assumptions of Proposition {\rm \ref{P:selectionLP}},
assume in addition that  $\theta$  is  a Bregman function with
zone $\R_{-}$ and that there exist $\alpha>0$, $\beta\geq 1$ and
$\eps>0$ such that
\begin{equation}
\label{comp} \forall s\in (0,\eps), \: \alpha
D_{\theta}(0,s)^{\beta}\leq s.
\end{equation}
Then there exists positive constants $K,L,M$ such that for all
$t>0$ the trajectory of $(H$-$SD)$ satisfies
$D_h(a,x(t))\leq Ke^{-Lt}$ if $\beta=1$, and $D_h(a,x(t))\leq
M/t^{\frac{1}{\beta-1}}$  if $\beta>1$.
\end{lemma}
\begin{proof} By Lemma \ref{L:norm},  there exists $k_0$ such that
for all $t>0$,
\begin{equation}
\label{restim} \langle c,x(t)-a\rangle \geq\sum
_{i\in I} k_0 |\langle B_i, x(t) \ra -\<B_i,a\rangle|.
\end{equation}  Now, if  we  prove that $\exists \lambda>0$ such that
\begin{equation}
\label{restim2} |\langle B_i, x(t) \ra -\<B_i,a\rangle|\geq
\lambda
D_{\theta}(\<B_i,a\rangle-d_i, \langle B_i, x(t) \ra -d_i)
\end{equation}
for all $i\in I$ and for $t$ large enough, then from
(\ref{restim}) it follows that $f(\cdot)=\< c,\cdot\ra$ satisfies
the assumptions of Proposition \ref{P:rate1}  and the conclusion follows
easily.
Since $x(t)\rightarrow a$, to prove (\ref{restim2}) it suffices
to show that $\forall r_0 \geq 0$, $\exists \eta, \mu>0$ such that
$\forall s$, $|s-r_0|<\eta$,
$\mu D_{\theta}(r_0,s)^{\beta} \leq |r_0-s|.$
The case where $r_0=0$ is a direct consequence of (\ref{comp}).
Let $r_0>0$. An easy computation yields $\frac{d^2}{ds^2}
D_{\theta}(r_0, s)_{|s=r_0}=\theta '' (r_0),$ and by Taylor's
expansion formula
\begin{equation}
\label{Taylor} D_{\theta}(r_0, s)=\frac{\theta ''
(r_0)}{2}(s-r_0)^2+o(s-r_0)^2
\end{equation}
with  $\theta '' (r_0)>0$ due to $(H_1)$(iii). Let $\eta$ be such
that $\forall s$, $|s-r_0|<\eta$, $s>0$, $D_{\theta}(r_0, s) \leq
\theta '' (r_0)(s-r_0)^2$ and $D_{\theta}(r_0, s) \leq  1$; since
$\beta \geq 1$, $D_{\theta}(r_0, s) ^{\beta}\leq D_{\theta}(r_0,
s) \leq \theta '' (r_0)|s-r_0|$. \end{proof}

To obtain  Euclidean estimates, the functions $s\mapsto
D_{\theta}(r_0, s)$, $r_0 \in \R _+$ have to be locally compared
to $s\mapsto |r_0-s|$. By (\ref{Taylor}) and the fact that
$\theta '' >0$, for each $r_0>0$ there exists $K,\eta>0$ such
that $ |r_0-s|\leq K\sqrt{D_{\theta}(r_0, s)}, \:\forall s,
\:|r_0-s|< \eta.$ This shows that, in practice, the Euclidean
estimate depends only on a property of the type
(\ref{comp}).  Examples:\\
$\bullet$ The Boltzmann-Shannon entropy  $\theta_3 (s)=s\ln(s)-s$ and
$\theta_6 (s)=s\ln s$ satisfy $D_{\theta _i}(0,s)=s$, $s>0$; hence
  for some $K,L>0$, $|x(t)-a|\leq K e^{-Lt}$, $\forall t\geq0$.\\
$\bullet$ With either $\theta_4 (s)= -s^{\gamma}/\gamma$ or
$\theta _5 (s)= (\gamma s-s^{\gamma})/(1-\gamma)$,
$\gamma \in (0,1)$, we have $D_{\theta_i}(0,s)=(1+1/\gamma)s^{\gamma}$,
$s>0$;
hence $|x(t)-a|\leq K/t^{\frac{\gamma}{2-2\gamma}}$, $\forall t>0.$




\subsection{Dual convergence}\label{S:dual}
In this section we focus on the  case $C=\R^n_{++}$, so that the
minimization problem is
$$
\min \{f(x)\mid x\geq 0,\; Ax=b\}. \leqno (P)
$$
We assume  \begin{equation}\label{A1}
\hbox{$f$ is
convex  and }S(P)\neq \emptyset,
\end{equation}
together with the Slater condition
\begin{equation}\label{A2}
\exists x^0\in \R^n,\; x^0>0,\; Ax^0=b.
\end{equation}
In convex
optimization theory, it is usual to associate with $(P)$ the {\em
dual} problem given by
$$
\min \{p(\lambda)\mid \lambda\geq 0\}, \leqno (D)
$$
where $p(\lambda)=\sup\{\<\lambda,x\ra-f(x)\mid Ax=b\}$. For many
applications,  dual solutions are as important as primal ones. In
the particular case of a linear program where $f(x)=\<c,x\ra$ for
some $c\in\R^n$, writing $\lambda=c+A^Ty$ with $y\in\R^m$ the
linear dual problem may  equivalently be expressed as $ \min
\{\<b,y\ra\mid A^Ty+c\geq 0\}$. Thus, $\lambda$ is interpreted as
a vector of {\em slack} variables for the dual inequality
constraints. In the general case,  $S(D)$ is nonempty and bounded
under (\ref{A1}) and (\ref{A2}), and moreover
\(S(D)=\{\lambda\in\R^n\mid \lambda\geq 0,\; \lambda\in\nabla
f(x^*)+\Im A^T,\; \<\lambda,x^*\ra=0\}\), where $x^*$ is any
solution of $(P)$; see for instance \cite[Theorems VII.2.3.2 and
VII.4.5.1]{HiL96}.\\
Let us introduce  a Legendre kernel $\theta$ satisfying $(H_1)$ and
define
  \begin{equation}\label{E:hdual}
h(x)=\sum_{i=1}^n \theta(x_i).
\end{equation}
Suppose that $\HS$ is well-posed. Integrating the differential
inclusion (\ref{E:cont-prox}), we obtain
\begin{equation}\label{E:opt-dual}
\lambda(t)\in c(t) +{\rm Im}A^T,
\end{equation}
where $
c(t)=\frac{1}{t}\int_0^t \nabla f(x(\tau))d\tau$
  and $\lambda(t)$ is the {\em dual trajectory } defined by
\begin{equation}\label{E:dual-traj}
\lambda(t)=\frac{1}{t}[\nabla h(x^0)-\nabla h(x(t))].
\end{equation}
Assume that $x(t)$ is bounded. From (\ref{A1}), it follows that
$\nabla f$ is constant on $S(P)$, and then  it is easy to see that
$\nabla f(x(t))\to\nabla f(x^*)$ as $t\to+\infty$ for any $x^*\in
S(P)$. Consequently,  $c(t)\to\nabla f(x^*)$. By
(\ref{E:dual-traj}) together with \cite[Theorem 26.5]{Roc70}, we have
$
x(t)=\nabla h^*(\nabla h(x^0)-t\lambda(t)),$ where the Fenchel
conjugate $h^*$ is given by $
h^*(\lambda)=\sum_{i=1}^n \theta^*(\lambda_i). $
Take any solution $\widetilde{x}$ of $A\widetilde{x}=b$. Since
$Ax(t)=b$, we have $\widetilde x-\nabla h^*(\nabla
h(x^0)-t\lambda(t))\in \Ker A$. On account of (\ref{E:opt-dual}),
$\lambda(t)$
is the unique optimal solution of
\begin{equation}\label{E:penalty}
\lambda(t)\in \argmin\left\{\<\widetilde{x},\lambda\ra+\frac{1}{t}
\sum_{i=1}^n \theta^*(\theta'(x^0_i)-t\lambda_i)\mid \lambda\in c(t)+{\rm
Im}A^T \right\}.
\end{equation}
By $(H_1)$(iii), $\theta'$ is increasing in $\R_{++}$. Set \(
\eta=\lim_{s\to +\infty}\theta'(s)\in (-\infty,+\infty]\). Since
$\theta^*$ is a Legendre type function, $\inte\dom
\theta^*=\dom\partial\theta^*=\Im \partial
\theta=(-\infty,\eta)$. From $(\theta^*)'=(\theta')^{-1}$, it
follows that $\lim_{u\to-\infty}(\theta^*)'(u)=0$ and
$\lim_{u\to\eta^-}(\theta^*)'(u)=+\infty$. Consequently,
(\ref{E:penalty}) can be interpreted as a {\em penalty
approximation scheme} of the dual problem $(D)$, where the dual
positivity constraints are penalized by a separable strictly
convex function. Similar schemes have been treated in
\cite{ACH97,Com,IuM00}. Consider the additional condition
\begin{equation}\label{A3}
\hbox{Either } \theta(0)<\infty, \hbox{ or } S(P) \hbox{ is
bounded, or $f$ is linear.}
\end{equation}
As a direct consequence of \cite[Propositions 10 and 11]{IuM00},
we obtain that under {\rm (\ref{A1}), (\ref{A2}), (\ref{A3})}
and   $(H_1)$, $\{\lambda(t)\mid t\to+\infty\}$ is bounded and
its cluster points belong to $S(D)$. The convergence of
$\lambda(t)$ is more difficult to establish. In fact, under some
additional conditions on $\theta^*$ (see \cite[Conditions
$(H_0)$-$(H_1)$]{Com} or \cite[Conditions (A7) and (A8)]{IuM00})
it is possible to show that $\lambda(t)$ converges to a
particular element of the dual optimal set (the
``$\theta^*$-center" in the sense of \cite[Definition 5.1]{Com} or
the $D_h(\cdot,x^0)$-center as defined in \cite[pag.
616]{IuM00}), which is characterized as the unique solution of a
{\em nested hierarchy} of optimization problems on the dual
optimal set. We will not develop this point here. Let us only
mention that  for all the examples of section \ref{S:Examples},
$\theta_i^*$ satisfies such additional conditions and
consequently:
\begin{proposition} Under {\rm (\ref{A1}), (\ref{A2})} and {\rm (\ref{A3})}, for 
each of the explicit Legendre kernels  given in  section
{\rm \ref{S:Examples}},  $\lambda(t)$ given by {\rm
(\ref{E:dual-traj})}  converges to a particular dual solution.
\end{proposition}

\section{Legendre transform coordinates}\label{S:legendre-transform}
\subsection{Legendre functions on affine subspaces}

The first objective of  this section is to slightly generalize
the notion of Legendre type function to the case of functions
whose domains  are contained in an affine subspace of $\RR^n$. We
begin by noticing that the Legendre type property does not depend
on canonical coordinates.

\begin{lemma}\label{L:leg-coord} Let $g\in \Gamma_0(\RR^r)$, $r\geq 1$,  and
$T:\RR^r\to\RR^r$ an
affine invertible mapping. Then $g$ is of  Legendre type iff
$g\circ T$ is of Legendre type.
\end{lemma}
\begin{proof} The proof is elementary and is left to the reader.
\end{proof}


  From now on,  $\A$ is the affine subspace  defined by
(\ref{E:affine-space}), whose dimension is $r=n-m$.

\begin{definition}\label{D:glegendre} A function $g\in\Gamma_0(\A)$ is said to
be of {\em Legendre type} if there exists an affine invertible
mapping $T:\A\to\RR^r$  such that $g\circ T^{-1}$ is a Legendre
type function in $\Gamma_0(\R^r)$.
\end{definition}

By Lemma \ref{L:leg-coord}, the previous  definition is
consistent.
\begin{proposition}\label{P:leg-trans2} Let $h\in\Gamma_0(\RR^n)$ be a function of
Legendre type
with $C=\inte\dom h$. If $\F=C\cap\A\neq \emptyset$ then the
restriction $h_{|_\A}$  of $h$ to $\A$ is of Legendre type and
moreover
$ {\rm int}_\A \dom h_{|_\A}=\F$ (${\rm int}_\A B$
stands for the interior of $B$ in $\A$ as a topological subspace
of $\RR^n$).
\end{proposition}
\begin{proof} From the inclusions $\F\subset \dom
h_{|_\A}\subset\overline{\F}=\overline{C}\cap\A$ and since
$\ri\overline{\F}=\F$, we conclude that $ {\rm int}_\A \dom
h_{|_\A}=\F\neq\emptyset$. Let $T:\RR^r\to\A$ be an invertible
transformation with $ Tz=Lz+x^0$ for all $z\in\RR^r$, where
$x^0\in\A$ and $L:\RR^r\to\A_0$ is a nonsingular linear mapping.
Define $k=h_{|_\A}\circ T$. Clearly, $k\in\Gamma_0(\RR^r)$. Let
us prove that $k$ is essentially smooth. We have $\dom
k=T^{-1}\dom h_{|_\A}$ and therefore $\inte\dom k=T^{-1}\F$. Since
$h$ is differentiable on $C$, we conclude that $k$ is
differentiable on $\inte\dom k$. Now, let $(z^j)\in\inte\dom k$
be a sequence that converges to a boundary point $z\in \bd\dom
k$. Then, $Tz^j\in{\rm int}_\A\dom h_{|_\A}$ and $Tz^j\to Tz\in
{\rm bd}_\A\dom h_{|_\A}\subset \bd\dom h$. Since $h$ is
essentially smooth, $|\nabla h(Tz^j)|\to +\infty$. Thus, to prove
that $|\nabla k(z^j)|\to +\infty$ it suffices to show that there
exists $\lambda>0$ such that $ |\nabla k(z^j)|\geq \lambda
|\nabla h(Tz^j)|$ for all $j$ large enough. Note that $ \nabla k
(z^j)=\nabla[h_{|_\A}\circ T](z^j)=L^*\nabla
h_{|_\A}(Tz^j)=L^*\Pi_{\A_0}\nabla h(Tz^j),$ where
$L^*:\A_0\to\RR^r$ is defined by $ \langle z,L^*x\rangle=\langle
Lz,x\rangle$, $\forall (z,x)\in \RR^r\times \A_0.$ Of course,
$L^*$ is linear with $\Ker L^*=\{0\}$. Therefore $ \frac{\nabla
k(z^j)}{|\nabla h(Tz^j)|}=L^*\Pi_{\A_0}\frac{\nabla
h(Tz^j)}{|\nabla h(Tz^j)|}.$ Let $\omega$ denote the nonempty and
compact set of cluster points of the normalized sequence $\nabla
h(Tz^j)/|\nabla h(Tz^j)|$, $j\in\mathbb{N}$. By Lemma
\ref{L:normal}, we have that $ \omega\subset\{\nu \in
N_{\overline{C}}(Tz)\: | \: |\nu|=1\},$ and consequently Lemma
\ref{L:lem1} yields $ \Pi_{\A_0}\omega\cap\{0\}=\emptyset.$ By
compactness of $\omega$, we obtain  $
\liminf_{j\to+\infty}|\Pi_{\A_0}\nabla h(Tz^j)|/|\nabla
h(Tz^j)|>0,$ which proves our claim. Finally, the strict
convexity of $k$ on $\dom \partial k=\inte\dom k=T^{-1}\F$ is a
direct consequence of the strict convexity of $h$ in $\F$. \end{proof}

\subsection{Legendre transform coordinates}

The prominent fact of Legendre functions theory is that
$h\in\Gamma_0(\RR^n)$ is of Legendre type iff its Fenchel 
conjugate $h^*$ is of Legendre type \cite[Theorem26.5]{Roc70}, and  
$\nabla h:\inte \dom h \to \inte\dom h^*$ is onto with 
$(\nabla h)^{-1}=\nabla h^*$. In
the case of Legendre functions on affine subspaces, we have the
following generalization:
\begin{proposition}\label{P:leg-trans1} If $g\in \Gamma_0(\A)$
is of Legendre type in the sense of Definition {\rm
\ref{D:glegendre}}, then $\nabla g({\rm int}_\A\dom g)$ is a
nonempty, open and convex subset of $\A_0$. In addition, $\nabla
g$ is a one-to-one continuous mapping from ${\rm int}_\A\dom g$
onto its image.
\end{proposition}
\begin{proof} Let $Tx=Lx+z_0$ with $L:\A_0\to\RR^r$ being a linear
invertible mapping and $z_0\in\RR^p$. Set $k=g\circ
T^{-1}\in\Gamma_0(\RR^r)$, which is of Legendre type. We have
$\dom k=T\dom g$. Define $L^*:\RR^r\to \A_0$ by $ \langle
L^*z,x\rangle=\langle z,Lx\rangle$, $\forall (z,x)\in
\RR^r\times\A_0$.  We have that $ \nabla g(x)=\nabla[k\circ
T](x)=L^*\nabla k(Tx)$ for all $x\in\inteA\dom g$. Therefore
$\nabla g(\inteA\dom g) =L^*\nabla k(T\inteA\dom g)=L^*\nabla
k({\rm int}_{\RR^r}\dom k)=L^*{\rm int}_{\RR^r}\dom k^*.$ Since
${\rm int}_{\RR^r}\dom k^*$ is a nonempty, open and convex subset
of $\RR^r$ and $L^*$ is an invertible linear mapping, then
$L^*{\rm int}_{\RR^r}\dom k^*$ is an open and nonempty subset of
$\A_0$. Moreover, by \cite[Theorem 6.6]{Roc70}, we have $L^*{\rm
int}_{\RR^r}\dom k^*=\ri L^*\dom k^*.$ Consequently, $ \nabla
g(\inteA\dom g)=\ri L^*\dom k^*={\rm int}_{\A_0} L^*\dom k^*\neq
\emptyset.$ Finally, since $\nabla k:{\rm int}_{\RR^r}\dom
k\to{\rm int}_{\RR^r}\dom k^*$ is one-to-one and continuous, the
same result holds for $\nabla g=L^*\circ\nabla k\circ T$ on
$\inteA\dom g$. \end{proof}

\vspace{2ex}

In the sequel, we assume that $h$ satisfies the basic condition
$(H_0)$ and $\F=C\cap\A\neq\emptyset$. The {\em Legendre
transform coordinates mapping} on $\F$ associated with $h$ is
defined by
\begin{equation}\label{E:LTC}
\begin{array}{cccl}
\phi_h:&\F&\to&\F^*=\phi_h(\F)\\
&x&\mapsto&\phi_h(x)=\nabla(h_{|_\A})=\Pi_{\A_0}\nabla h(x).
\end{array}
\end{equation}
This definition
retrieves the Legendre transform coordinates introduced by Bayer
and Lagarias in \cite{BaL89} for the particular case of the
log-barrier on a polyhedral set.

\begin{theorem}\label{T:dphi} Under the above definitions and assumptions,  
$\F^*$ is a
convex, (relatively) open and
nonempty subset of $\A_0$, $\phi_h$ is a ${\cal C}^1$
diffeomorphism from $\F$ to $\F^*$, and for all $x\in\F$,
$d\phi_h(x)=\Pi_{\A_0}H(x)$ and
$d\phi_h(x)^{-1}=\sqrt{H(x)^{-1}}\Pi_{\sqrt{H(x)}\A_0}\sqrt{H(x)^{-1}}$,
where $H(x)=\nabla^2h(x)$.
\end{theorem}
\begin{proof} By Propositions \ref{P:leg-trans2} and
\ref{P:leg-trans1}, $\F^*$  is a convex, open and nonempty subset
of $\A_0$ and $\phi_h$ is a continuous  bijection. By
$(H_0)$(ii), $\phi_h$ is of class ${\cal C}^1$ on $\F$ and we have for
all $x\in \F$, $d\phi_h(x)=\Pi_{\A_0}\nabla ^2 h(x)=\Pi_{\A_0}
H(x).$  Let $v\in\A_0$ be such that $d\phi_h(x)v=0$. It follows
that $H(x)v\in\A_0^\perp$ and in particular $\langle
H(x)v,v\rangle=0$. Hence, $v=0$ thanks to $(H_0)$(iii). The
implicit function theorem implies then that $\phi_h$ is a ${\cal C}^1$
diffeomorphism. The  formula concerning $d\phi_h(x)^{-1}$ is a
direct consequence of the next lemma.
\begin{lemma} Define the linear operators $L_i:\RR^n\to\RR^n$ by
$L_1=\Pi_{\A_0} H(x)$ and
$L_2=\sqrt{H(x)^{-1}}\Pi_{\sqrt{H(x)}\A_0}\sqrt{H(x)^{-1}}$. Then
$L_2L_1v=v$ for all $v\in\A_0$.
\end{lemma}

This  follows by the same method as in \cite{BaL89}, pag. 545; we
leave the proof to the reader. \end{proof}

Similarly to the classical Legendre type functions
theory, the
inverse of $\phi_h$ can be expressed in terms of Fenchel
conjugates. For that purpose, we notice that inverting $\phi_h$ is a
minimization problem. Indeed,   given $y\in\A_0$,  the problem of
finding $x\in\F$ such that $y=\Pi_{\A_0}\nabla h(x)$ is
equivalent to  $ x=\argmin\{h(z)-\<y,z\ra|z\in\A\}$, or
equivalently
\begin{equation}\label{E:inv-min}
x=\argmin\{(h+\delta_{\A})(z)-\<y,z\ra\},
\end{equation}
where $\delta_{\A}$ is the {\em indicator} of $\A$, i.e.
$\delta_{\A}(z)= 0$ if $z\in \A$ and $+\infty$ otherwise.  Let us
recall the definition of {\em epigraphical
sum} of two functions  $g_1,g_2\in\Gamma_0(\RR^n)$, which is given by
  $ \left(g_1\square g_2\right)(y)=\inf\{ g_1(u)+g_2(v)|u+v=y\}$, $\forall
y\in\RR^n.$
We have $g_1\square g_2\in\Gamma_0(\RR^n)$ and
if $g_1$ and $g_2$ satisfy
  $\ri\dom g_1\cap \ri\dom g_2\neq \emptyset$  then
$(g_1+g_2)^*=g_1^*\square
  g_2^*$ (see \cite{Roc70}).
\begin{proposition}\label{P:F*} We have that $\phi_h^{-1}:\F^*\to\F$ is given by $
\phi_h^{-1}(y)=\nabla[h^*\square(\delta_{\A_0^\perp}+\<\cdot,\widetilde
x\ra)](y),$ for any $\widetilde x\in\A$, and moreover \(
\F^*=\Pi_{\A_0}\inte\dom h^*\).

\end{proposition}
\begin{proof} The optimality condition for (\ref{E:inv-min})
yields $y\in\partial(h+\delta_{\A})(x).$ Thus, $x\in\partial
(h+\delta_\A)^*(y)$. From $\F\neq\emptyset$, we conclude that the
function $g\in\Gamma_0(\RR^n)$ defined by $g=(h+\delta_\A)^*$
satisfies $g =h^*\square\delta_\A^*=h^*\square
(\delta_{\A_0^\perp}+\<\cdot,\widetilde x\ra)$ with $\widetilde
x\in\A$. Moreover, by \cite[Corollary 26.3.2]{Roc70},  $g$ is
essentially smooth and we deduce that indeed $x=\nabla g(y)$.
Since $g$ is essentially smooth,  $\dom\partial g=\inte\dom g$.
By definition of epigraphical sum, $
g(y)=\inf\{h^*(u)+\delta_{\A_0^\perp}(v)+\<v,\widetilde x\ra
|u+v=y\},$ and consequently we have that $y\in \dom g$ iff $y\in
\dom h^*+ \A_0^\perp$. Hence, $\inte\dom g=\inte\dom h^* +
\A_0^\perp$ (see for instance \cite[Corollary 6.6.2]{Roc70}).
Recalling that $\F^*$ is a relatively open subset of $\A_0$, we
deduce that $\F^*=\Pi_{\A_0}\dom\partial g=\Pi_{\A_0}\inte\dom
h^*$. \end{proof}

\subsection{Linear problems in Legendre transform coordinates} 
\subsubsection{Polyhedral sets in Legendre transform coordinates} One of the first interest of Legendre transform coordinates is to transform linear constraints into
 positive cones.
\begin{proposition}\label{P:domh*} Assume that $ C=\{x\in\RR^n | Bx>d\}$, where $B$ 
is  a
$p\times n$ full rank matrix, with $p\geq n$. Suppose also that $h$ is of the form {\rm
(\ref{E:defhLP})} with $\theta$ satisfying $(H_1)$, and let
$\eta=\lim_{s\to+\infty}\theta'(s)\in(-\infty,+\infty]$. 
If
$\eta<+\infty$ then \( \overline{\dom h^*}=\{y\in\R^n\mid
y+B^T\lambda=0,\;\lambda_i\geq -\eta\}, \) and $\dom h^*=\RR^n$
when $\eta=+\infty$.
\end{proposition}
\begin{proof} By \cite[Theorem 11.5]{RoW98}, $\overline{\dom
h^*}=\{y\in\R^n\mid \<y,d\ra\leq h^\infty(d) \hbox{ for all
}d\in\R^n\}$, where $h^\infty$ is the {\em recession} function,
also known as {\em horizon} function,  of $h$. The recession
function is defined by
$h^\infty(d)=\lim_{t\to+\infty}\frac{1}{t}[h(\bar x+td)-h(\bar
x)], \; d\in\RR^n$, where $\bar x\in\dom h$; this limit does not
depend of $\bar x$ and eventually $h^\infty(d)=+\infty$ (see also
\cite{Roc70}). In this case, it is easy to verify that $
h^\infty(d)=\sum_{i=1}^p\theta^\infty(\<B_i,d\ra).$ Clearly,
$\theta^\infty(-1)=+\infty$ and
$\theta^\infty(1)=\lim_{s\to+\infty}\theta'(s)=\eta$. In
particular, if $\eta=+\infty$ then $\dom h^*=\RR^n$. If
$\eta<+\infty$ then $y\in\overline{\dom h^*}$ iff for all
$d\in\RR^n$ such that $Bd\geq 0$, $\<y,d\ra\leq
h^\infty(d)=\sum_{i=1}^p\eta\<B_i,d\ra$, that is $\<y-\eta
B^Te,d\ra\leq 0 $ with $e=(1,\cdots,1)$. Thus, by the Farkas
lemma, $y\in\overline{\dom h^*}$ iff $\exists \mu\geq 0$, $y-\eta
B^Te+B^T\mu=0$. \end{proof}

As a direct consequence of  Propositions
\ref{P:F*} and \ref{P:domh*}:
\begin{corollary} Under the assumptions of Proposition {\rm
\ref{P:domh*}}, if $\eta=0$ then $\F^*$ is a positive convex cone
and
if $\eta=+\infty$ then $\F^*=\A_0$.
\end{corollary}
\subsubsection{$\HS$-trajectories  in Legendre transform
coordinates} In the sequel, we assume that  $
f(x)=\langle c,x\rangle $ for some $c\in\RR^n$. As another striking application of Legendre transform coordinates,
 we prove now that the trajectories of $\HS$ may be seen as straight lines in
 $\F^*$. 
Recall that the  {\em push forward} vector field of $\gr f_{|_\F}$ by
$\phi_h$ is defined for every \(y\in\F^* \) by \([(\phi_h) _*\gr
f_{|_{\F}} ]\:(y)=d\phi_h(\phi_h^{-1}(y))\gr
f_{|_\F}(\phi_h^{-1}(y))\). 
\begin{proposition}\label{P:pushforward} For all $y\in\F^*$, $
[(\phi_h) _*\gr f_{|_\F} ]\:(y)=\Pi_{\A_0}c .$
\end{proposition}
\begin{proof} Let $y\in\F^*$. Setting $x=\phi_h^{-1}(y)$, by
Theorem \ref{T:dphi} we get $[(\phi_h) _*\gr f_{|_{\F}}]
\:(y)=d\phi_h(x)\gr f_{|_\F}(x)=
\Pi_{\A_0}H(x)H(x)^{-1}[I-A^T(AH(x)^{-1}A^T)^{-1}AH(x)^{-1}]c=
\Pi_{\A_0}c-\Pi_{\A_0}A^Tz,$ where $
z=[(AH(x)^{-1}A^T)^{-1}AH(x)^{-1}]c.$ Since $\Im A^T=\A_0^\perp$,
the conclusion follows. \end{proof}

   Next, we  give two optimality characterizations of the
orbits of $\HS$, extending thus to
the general case the results of \cite{BaL89} for the log-metric.

\subsubsection{Geodesic curves} First, we claim that  the orbits of $\HS$ 
can be regarded as  geodesics
curves with respect to some appropriate metric on $\F$. To this
end, we endow $\F^*=\phi_h(\F)$ with the Euclidean metric, which
allows us to define on $\F$ the  metric
\begin{equation}\label{E:H2}
(\cdot,\cdot)^{H^2}=\left( \phi_h \right)^* \<\cdot,\cdot\ra,
\end{equation}
that is, $\forall (x,u,v) \in \F\times\R^n\times\R^n$,
$(u,v)^{H^2} _x =\< d\phi_h (x)u,d\phi_h (x)v\ra=\<
\Pi_{\A_0}H(x)u,\Pi_{\A_0}H(x)v\ra.$  For each initial condition
$x^0\in \F$, and for every $c\in\R^n$ we set
\begin{equation}
\label{geovelo} v=d\phi_h
(x^0)^{-1}\Pi_{\A_0}c=\sqrt{H(x^0)^{-1}}\Pi_{\sqrt{H(x^0)}\A_0}\sqrt{H(x^0)^{-1}}\Pi_{\A_0}c.
\end{equation}
\begin{theorem}
Let $(x^0,c)\in \F\times\R^n$, set $f(x)=\<c,x\ra, \: \forall
x\in C$ and define $v$ as in {\rm (\ref{geovelo})}. If $\F$ is
endowed with the metric $(\cdot,\cdot)^{H^2}$  given by {\rm
(\ref{E:H2})}, then the solution $x(t)$ of $\HS$ is the unique
geodesic passing through $x^0$ with velocity $v$.
\end{theorem}
\begin{proof} Since $\F,\:(\cdot,\cdot)^{H^2}$ is isometric to the
Euclidean Riemannian manifold $\F^*$, the geodesic joining two
points of $\F$ exists and is unique. Let us denote by $\gamma:J
\subset \R \mapsto \F$ the geodesic passing through $x^0$ with
velocity $v$. By definition of $(\cdot,\cdot)^{H^2}$,
$\phi_h(\gamma)$ is a geodesic in $\F^*$. Whence
$\phi_h(\gamma(t))=\phi_h(x^0)+td\phi_h(x^0)v,$ where $t\in J$.
In view of (\ref{geovelo}), this can be rewritten
$\phi_h(\gamma(t))=\phi_h(x^0)+t\Pi_{\A_0}c$. By Proposition
\ref{P:pushforward} we know that $\left( \phi_h \right)_* \gr
f_{|\F}=\Pi_{\A_0}c$, and therefore
$\phi_h^{-1}(\phi_h(\gamma))=\gamma$ is exactly the solution of
$\HS$.\end{proof}

\begin{remark}
{\rm A Riemannian manifold is called {\em geodesically complete}
if the maximal interval of definition of every geodesic is
$\RR$.  When $\Pi_{\A_0}c\neq 0$ and $\F^*$ is not an affine
subspace of $\R^n$, the Riemannian manifold
$\F,\:(\cdot,\cdot)^{H^2}$ is not complete in this sense.}

\end{remark}

\subsubsection{Lagrange equations} Following the ideas of \cite{BaL89}, we 
describe  the orbits
of $\HS$ as orthogonal projections on \( \A \) of \( \dot{q}-
\)trajectories of a specific {\em Lagrangian system}. Recall that
given a real-valued mapping \( {\cal L}(q,\dot q) \) called the
Lagrangian, where  \( q=(q_{1},\dots ,q_{n}) \) and \( \dot
q=(\dot q_1,\dots ,\dot q_{n}) \), the associated Lagrange
equations of motion are the following
\begin{equation}\label{LAG}
\frac{d}{dt}\frac{\partial {\cal L}}{\partial
\dot q_{i}}=\frac{\partial {\cal L}}{\partial q_{i}}, \quad
\frac{d}{dt}q_{i}=\dot q_{i},\quad \forall i=1\dots n.
\end{equation}
Their solutions are \( C^{1}- \)piecewise paths \( \gamma
:t\longmapsto (q(t),\dot q(t)) \), defined for \( t\in J\subset
\R \), that satisfy (\ref{LAG}), and appear as extremals of the
functional \( \widehat{{\cal L}}(\gamma )=\int _{J}{\cal L}(q(t),\dot
q(t))dt\). Notice that in general, the solutions are not unique,
in the sense that they do not only depend on the initial
condition \( \gamma (0) \). Let us introduce the Lagrangian
${\cal L}:\RR^n\times C\to \RR$ defined  by
\begin{equation}\label{E:LAG} {\cal L} (q,\dot q)=\<\Pi 
_{\A_{0}}c,q\ra-h(\Pi
_{\A}\dot q),
\end{equation}
where \( \Pi _{\A}\) is the orthogonal projection onto \( \A \),
i.e.  \(\Pi _{\A}x=\widetilde x+\Pi _{\A_{0}}(x-\widetilde x) \)
for any \( \widetilde x\in \A \).

\begin{theorem}\label{T:LAG}
For any solution \( \gamma (t)=(q(t),\dot q(t)) \) of the
Lagrangian dynamical system {\rm (\ref{LAG})} with Lagrangian
given by {\rm (\ref{E:LAG})}, the projection \( x(t)=\Pi
_{\A}\dot q(t) \) is the solution of (H-SD)  with initial
condition \( x^0=\Pi _{\A}\dot{q}(0). \)
\end{theorem}
\begin{proof} It is easy to verify  that  \( \nabla (h\circ \Pi
_{\A})(x)=\Pi _{\A_{0}}\nabla h(\Pi _{\A}x) \) for any \( x\in
\R^{n}. \) Given a solution \( \gamma (t)=(q(t),\dot{q}(t)) \) of
(\ref{E:LAG}) defined on \( J \), we set \( p(t)=(p_{1}(t),\dots
,p_{n}(t))=\left(\frac{\partial {\cal L}}{\partial \dot{q}_{1}}(\gamma
(t)), \dots ,\frac{\partial {\cal L}}{\partial \dot{q}_{n}}(\gamma
(t))\right)\). We have \( p(t)=\nabla (h\circ \Pi
_{\A})(\dot{q}(t))=\Pi _{\A_{0}}\nabla h(\Pi
_{\A}\dot{q}(t))=\phi _{h}(\Pi _{\A}\dot{q}(t)).\)  Equations of
motion  become \( \frac{d}{dt}p(t)=\Pi _{\A_{0}}c,\) that is, \(
\frac{d}{dt}\phi _{h}(\Pi _{\A}\dot{q}(t))=\Pi _{\A_{0}}c \).
Since \( \phi _{h}:\F\to\F^* \) is a diffeomorphism,  the latter
means, according to Proposition \ref{P:pushforward}, that \(
\Pi _{\A}\dot{q}(t) \) is a trajectory for the vector field \(
\nabla _{H}f_{|_\F}. \) Notice that \( C \) being convex, as soon
as \( \dot{q}(0)\in C, \) \( \Pi _{\A}\dot{q}(0)\in C\cap \A=\F,
\) and what precedes forces \( \Pi _{\A}\dot{q}(t) \) to stay in
\( \F \) for any \( t\in J. \) \end{proof}

\subsubsection{Completely integrable Hamiltonian systems} In the sequel,
all mappings are supposed to be at least of class \( {\cal C}^{2} \).
Let us first recall the notion of Hamiltonian system. Given an integer \(r\geq 1\) and a real-valued mapping \( {\cal H}(q,p) \) on \( \R^{2r}\) with coordinates \((q,p)=(q_{1},\dots ,q_{r},p_{1},\dots ,p_{r}) \),  the  {\it Hamiltonian vector field \( X_{{\cal H}} \) associated with \( {\cal H} \)}  is defined by 
\(
X_{{\cal H}}=\sum _{i=1}^r \frac{\partial {\cal H}}{\partial
p_{i}}\frac{\partial }{\partial q_{i}}-\frac{\partial {\cal H}}{\partial
q_{i}}\frac{\partial }{\partial p_{i}}.\) The trajectories of the dynamical system induced by \( X_{{\cal H}} \)  are the solutions to 
\begin{equation}
\left\{ \begin{array}{l}
\dot{p}_{i}(t)=-\frac{\partial} {\partial q_{i}}{\cal
H}(q(t),p(t)),\:i=1,\dots, r,\\
\dot{q}_{i}(t)=\frac{\partial} {\partial p_{i}}{\cal
H}(q(t),p(t)),\:i=1,\dots, r.
\end{array}\label{Ham}\right.
\end{equation}
Following a standard procedure,  Lagrangian functions  \( {\cal L}(q,\dot{q}) \) are  associated with  Hamiltonian  systems by means of  the so-called Legendre transform
\[\Phi :\left\{
\begin{array}{ccl}
\R^{2r}& \longrightarrow &\R^{2r}\\
(q,\dot{q}) &\longmapsto & (q,\frac{\partial {\cal L}}{\partial
\dot{q}}(q,\dot{q}))
\end{array}\right.\]
In fact, when \( \Phi  \) is a diffeomorphism, the Hamiltonian function \( {\cal H} \) associated
with the Lagrangian \( {\cal L} \) is defined on \( \Phi(\R^{2r}) \)  by \(
{\cal H}(p,q)  = \sum _{i=1}^r p_{i}\dot{q}_{i}-{\cal L}(q,\dot{q}) =  \<p,\psi ^{-1}(q,p)\ra-{\cal L}(q,\psi ^{-1}(q,p)),\)
where \( (q,\psi ^{-1}(q,p)):=\Phi ^{-1}(q,p) \). With these definitions, \( \Phi  \) sends the
trajectories of the corresponding Lagrangian system  on the trajectories of the Hamiltonian system (\ref{Ham}).

In general, the Lagrangian (\ref{E:LAG})  does not  lead to an 
invertible \( \Phi  \) on \(\R^{2n}\). However, we are only interested in the projections \( \Pi _{\A}\dot{q} \) of
the trajectories, which, according to Theorem \ref{T:LAG}, take their values in 
\( \F \). Moreover,  notice that for any differentiable path \( t\mapsto q^{\perp }(t) \) lying in \(
\A_{0}^{\perp } \),
\( t\mapsto (q(t),\dot{q}(t)) \) is a solution of (\ref{LAG}) iff
\( t\mapsto (q(t)+q^{\perp }(t),\dot{q}(t)+\dot{q}^{\perp }(t)) \) is.
This legitimates the idea of restricting  \( {\cal L} \) to \(
\A_{0}\times \Pi _{\A_{0}}\F \). Hence and from now on,
\( {\cal L} \) denotes the function:
\begin{equation}\label{E:Lag2}
{\cal L}: \left\{
\begin{array}{ccl}
\A_{0}\times \Pi _{\A_{0}}\F& \longrightarrow & \R\\
(q,\dot{q})& \longmapsto & {\cal L}(q,\dot{q}).
\end{array}\right.
\end{equation}
Taking \( (q_{1},\dots, q_{r}) \), with $r=n-m$,  a linear system of coordinates induced by an
Euclidean orthonormal basis for \( \A_{0} \), we easily see that
this ``new'' Lagrangian has
trajectories \( (q(t),\dot{q}(t)) \) lying in \( \A_{0}\times \Pi
_{\A_{0}}\F \), whose projections \( \Pi _{\A}\dot{q}(t) \) are exactly the \(\HS\) trajectories. Moreover, an easy computation yields
\[
\frac{\partial {\cal L}}{\partial \dot{q}}(q,\dot{q})=\Pi 
_{\A_{0}}\nabla h(\Pi _{\A_{0}}\dot{q})=[\phi _{h}\circ \Pi _{\A}](\dot{q}),\]
which is a diffeomorphism by Proposition \ref{T:dphi}. The Legendre transform is then given by 
\[\Phi : \left\{
\begin{array}{ccl}
\A_{0}\times \Pi _{\A_{0}}\F & \longrightarrow & \A_{0}\times \F^{*}\\
(q,\dot{q})& \longmapsto & (q,[\phi _{h}\circ \Pi _{\A}](\dot{q})),
\end{array}\right.\]
and therefore, \( {\cal L} \) is converted into the Hamiltonian system 
associated with
\begin{equation}
{\cal H}:\left\{
\begin{array}{ccl}
\A_{0}\times \F^{*}& \longrightarrow & \R\\
(q,p)&\longmapsto & \<p,[\phi _{h}\circ \Pi _{\A}]^{-1}(p)\ra-{\cal 
L}(q,[\phi
_{h}\circ \Pi _{\A}]^{-1}(p)).
\end{array}\right.\label{Ham2}
\end{equation}
Let us now introduce the concept of completely integrable Hamiltonian
system.  The Poisson bracket of two real valued
functions \( f_{1},f_{2} \)
on \( \R^{2r}\) is given by
\(
\{f_{1},f_{2}\}=\sum _{i=1}^r\frac{\partial f_{1}}{\partial p_{i}}
\frac{\partial f_{2}}{\partial q_{i}}-\frac{\partial f_{1}}
{\partial q_{i}}\frac{\partial f_{2}}{\partial p_{i}}.\) Notice that, from the definitions, we have \(
\{f_{1},f_{2}\}=X_{f_{1}}(f_{2}) \)
and \( X_{\{f_{1},f_{2}\}}=[X_{f_{1}},X_{f_{2}}] \), where \([\cdot,\cdot]\) is the standard {\em bracket product} of vector fields \cite{Lan95}.  Now, the system
(\ref{Ham})
is called {\it completely integrable} if there exist $r$ functions \( f_{1},\dots 
,f_{r} \) with \( f_{1}={\cal H} \),  satisfying
\[
\left\{ \begin{array}{l}
\{f_{i},f_{j}\}=0, \quad \forall i,j=1,\dots, r.\\
df_{1}(x),\dots,df_{r}(x) \textrm{ are linearly independent at any }x\in \R^{2r}.
\end{array}\right. \]
As a motivation for completely integrable systems, we will just point out
the
following: the functions \( f_{i} \) are called {\it integrals of motions}
because
\( X_{{\cal H}}(f_{i})=\{h,f_{i}\}=0 \), which means that any trajectory of
\( X_{{\cal H}} \)
lies on the level sets of each \( f_{i} \) (the same holds for all \(
X_{f_{j}} \)).
Also,  the trajectory passing through \( (q_{0},p_{0}) \) lies in
  the set \( \bigcap _{i=1\dots r}f^{-1}_{i}(\{f_{i}(q_{o},p_{0})\}) \).
Besides,
\( [X_{f_{i}},X_{f_{j}}]=0 \) implies that we can find, at least locally,
coordinates
\( (x_{1},\dots ,x_{r}) \) on this set such that \( X_{{\cal H}
}=\frac{\partial }{\partial x_{1}},X_{f_{2}}=\frac{\partial }{\partial
x_{2}},
\dots ,X_{f_{r}}=\frac{\partial }
{\partial x_{r}}, \)
that is, in these coordinates, the trajectories of \( X_{f_i} \) are
straight
lines.


\begin{theorem} Suppose $\Pi_{A_0}c\neq 0.$ The Lagrangian system  on \( \A_{0}\times \Pi _{\A_{0}}\F \) associated with {\rm (\ref{E:LAG}), (\ref{E:Lag2})}
gives rise, by the Legendre transform, to a completely integrable
Hamiltonian system on \( \A_{0}\times \F^{*} \) with Hamiltonian given by {\rm (\ref{Ham2})}. 
\end{theorem}
\begin{proof} There only remains to prove the complete integrability of the system. To this end, we adapt the proof of \cite[Theorem II.12.2]{BaL89} to our abstract framework. Take the integrals of motion to be $f_{1}={\cal H}$, $f_{i}(q,p)=\<c_{i},p\ra,\:i=2,\dots, r$ where $r=n-m$ and  \(\{ \Pi _{\A_{0}}c,\:c_{2},\dots ,c_{r} \}\) is chosen as to be an
orthonormal basis of \( \A_{0} \). For any \( i,j\in\{2,\dots, r\}, \) \( \{f_{i},f_{j}\} \) is
zero since \( f_{i} \) and \( f_{j} \) only depend on \( p \).  Let \( \phi
_{h,l}^{-1}(q,p) \) (resp. \( (\Pi _{\A_{0}}c)_{l} \)) stand
for the \( l \)-th component of \( \phi_h^{-1} (q,p) \) (resp. the \( l 
\)-th
component
of \( \Pi _{\A_{0}}c \)) and take some $k\in\{1,...,r\}$. Since
\begin{eqnarray*}
\frac{\partial {\cal H}}{\partial q_{k}}(q,p)&=&\frac{\partial (\sum
_{l=1}^r p_{l}\phi ^{-1}_{h,l})}{\partial q_{k}}(q,p)-\frac{\partial ({\cal 
L}\circ
\Phi ^{-1})}{\partial q_{k}}(q,p)\\
&=&\sum _{l=1}^r p_{l}\frac{\partial \phi _{h,l}^{-1}}
{\partial q_{k}}(p,q)-\frac{\partial {\cal L}}{\partial q_{k}}(q,\phi_h 
^{-1}(q,p))
-\sum _{l=1}^r\frac{\partial {\cal L}}{\partial \dot{q}_{l}}
(q,\phi_h ^{-1}(q,p))\frac{\partial \phi _{h,l}}{\partial q_{k}}(q,p)\\
&=&-(\Pi _{\A_{0}}c)_{k}
\end{eqnarray*}
we deduce that for all $i\in \{2,...,r\}$, \(\{{\cal H},f_{i}\}  = \sum _{k=1}^r-\frac{\partial f_{i}}{\partial
p_{k}}\frac{
\partial {\cal H}}{\partial q_{k}}= \<\Pi _{\A_{0}}c,c_{i}\ra=0\).
The second condition for complete integrability is satisfied too, as the
 \(r\times 2r \) matrix
\[
\left(
[\frac{\partial f_{i}}{\partial q_{1}},\dots,\frac{\partial f_{i}}{\partial q_{r}},  \frac{\partial f_{i}}{\partial
p_{1}},\dots,\frac{\partial f_{i}}{\partial
p_{r}}]
\right)_{i=1,\dots,r}=\left( \begin{array}{cc}
\Pi _{\A_{0}}c^T & \star \\
0 & \begin{array}{c}
c_{1}^T\\
\dots \\
c_{r}^T
\end{array}
\end{array}\right) \]
is full rank. \end{proof}



\end{document}